\documentclass[11pt]{amsart}

\usepackage{amsfonts}
\usepackage{amsmath}
\usepackage{amsthm}
\usepackage{url}
\usepackage[all]{xy}
\usepackage{graphicx}
\usepackage{latexsym}
\usepackage{amssymb}
\usepackage[cp850]{inputenc}
\usepackage{epsfig}
\usepackage{psfrag}
\usepackage{amsfonts}

\newtheorem*{mthm}{Main Theorem}
\newtheorem{sat}{Theorem}[section]		
\newtheorem{lem}[sat]{Lemma}
\newtheorem{kor}[sat]{Corollary}			\newtheorem{prop}[sat]{Proposition}
				
\newtheorem*{defi*}{Definition}			\newtheorem*{bei*}{Example}
\newtheorem*{sat*}{Theorem}				\newtheorem*{kor*}{Corollary}
\newtheorem*{rmk*}{Remark}

\let\ssection=\section
\renewcommand{\section}{\setcounter{equation}{0}\ssection}

\newtheorem*{namedtheorem}{\theoremname}
\newcommand{\theoremname}{testing}
\newenvironment{named}[1]{\renewcommand{\theoremname}{#1}\begin{namedtheorem}}{\end{namedtheorem}}

\theoremstyle{remark}
\newtheorem*{bem}{Remark}

\newcommand{\BC}{\mathbb C}			\newcommand{\BH}{\mathbb H}
\newcommand{\BR}{\mathbb R}			
			\newcommand{\BQ}{\mathbb Q}
\newcommand{\BS}{\mathbb S}			\newcommand{\BZ}{\mathbb Z}
				\newcommand{\BT}{\mathbb T}
			\newcommand{\BG}{\mathbb G}

		\newcommand{\CB}{\mathcal B}
		
		\newcommand{\CF}{\mathcal F}

\newcommand{\CQ}{\mathcal Q}		
\newcommand{\CS}{\mathcal S}

\newcommand{\actson}{\curvearrowright}
\newcommand{\D}{\partial}

\newcommand{\bs}{\backslash}

\DeclareMathOperator{\SL}{SL}		
\DeclareMathOperator{\PSL}{PSL}		
\DeclareMathOperator{\GL}{GL}		
\DeclareMathOperator{\Id}{Id}		

\DeclareMathOperator{\rank}{rank}

\newcommand{\comment}[1]{}

\DeclareMathOperator{\SO}{SO}
\DeclareMathOperator{\OO}{O}

\DeclareMathOperator{\vcd}{{vcd}}

\DeclareMathOperator{\sing}{sing}
\DeclareMathOperator{\Sp}{Sp}
\DeclareMathOperator{\UU}{U}
\DeclareMathOperator{\gdim}{\underline{gd}}
\DeclareMathOperator{\cdm}{\underline{cd}}
\DeclareMathOperator{\SU}{SU}
\DeclareMathOperator{\Gr}{Gr}
\DeclareMathOperator{\Span}{Span}

\begin{document}

\title[]{Geometric dimension of lattices in classical simple Lie groups}
\author[J. Aramayona, D. Degrijse, C. Martinez-Perez, J. Souto]{Javier Aramayona, Dieter Degrijse, Conchita Martinez-Perez, Juan Souto}

\begin{abstract}
We prove that if $\Gamma$ is a lattice in a classical simple Lie group $G$, then the symmetric space of $G$ is $\Gamma$-equivariantly homotopy equivalent to a proper cocompact $\Gamma$-CW complex of dimension the virtual cohomological dimension of $\Gamma$.
\end{abstract}
\maketitle
\section{Introduction}
Let $\Gamma$ be an infinite discrete group. A $\Gamma$-CW-complex $X$ is said to be a model for $\underline E\Gamma$, or a {\em classifying space for proper actions}, if the stabilizers of the action of $\Gamma$ on $X$ are finite and for every finite subgroup $H$ of $\Gamma$, the fixed point space $X^H$ is contractible. Note that any two models for $\underline{E}\Gamma$ are $\Gamma$-equivariantly homotopy equivalent to each other (see \cite[Def. 1.6 \& Th. 1.9]{Luck2}). A model $X$ is called \emph{cocompact} if the orbit space $\Gamma \setminus X$ is compact. The {\em proper geometric dimension} $\gdim(\Gamma)$ of $\Gamma$ is  by definition the smallest possible dimension of a model of $\underline E\Gamma$. We refer the reader to the survey paper \cite{Luck2} for more details and terminology about these spaces. 

The aim of this paper is to compare the geometric dimension $\gdim(\Gamma)$ of certain virtually torsion-free groups $\Gamma$ with their {\em virtual cohomological dimension} $\vcd(\Gamma)$. Recall that $\vcd(\Gamma)$ is the cohomological dimension of a torsion-free finite index subgroup of $\Gamma$. Due to a result by Serre, this definition does not depend of the choice of finite index subgroup (see \cite[Ch. VIII. Sec. 3]{brown}). In general one has 
$$\vcd(\Gamma)\le\gdim(\Gamma)$$
but this inequality may be strict. Indeed, in \cite{LearyNucinkis} Leary and Nucinkis constructed examples of groups $\Gamma$ for which $\gdim(\Gamma)$ is finite but strictly greater than $\vcd(\Gamma)$. In fact, they show that the gap can be arbitrarily large. Recently, Leary and Petrosyan \cite{LP} were able to construct examples of virtually torsion-free groups $\Gamma$ that admit a  cocompact model for $\underline E\Gamma$ such that $ \vcd(\Gamma)< \gdim(\Gamma)$. Other examples of this nature were subsequently also given in \cite{DMP2}.

On the other hand, one has $\vcd(\Gamma)=\gdim(\Gamma)$ for many important classes of virtually torsion-free groups. For instance, equality holds for elementary amenable groups of type $\mathrm{FP}_\infty$ \cite{KMPN}, $\mathrm{SL}(n,\mathbb{Z})$ \cite{Ash,Luck2},  mapping class groups \cite{Armart},   outer automorphism groups of free groups \cite{Luck2,Vogtmann} and for groups that act properly and chamber transitively on a building, such as Coxeter groups and graph products of finite groups \cite{DMP}. In this paper, we will add lattices in classical simple Lie groups to this list. The classical simple Lie groups are the complex Lie groups 
$$\SL(n,\BC),\ \SO(n,\BC),\ \Sp(n,\BC)$$
and their real forms
$$\SL(n,\BR),\ \SL(n,\BH),\ \SO(p,q),\ \SU(p,q),\ \Sp(p,q),\ \Sp(2n,\BR),\ \SO^*(2n)$$
with conditions on $n$ and $p+q$ to ensure simplicity (see Section \ref{sec: classical groups}). Our main result is the following.

\begin{mthm}
If $\Gamma$ is a lattice in a classical simple Lie group, then $\gdim(\Gamma)=\vcd(\Gamma)$.
\label{main}
\end{mthm}

Note that if $\Gamma$ is a lattice in a simple Lie group $G$, and if $K\subset G$ is a maximal compact subgroup, then the corresponding symmetric space $S=G/K$ is a CAT(0) space on which $\Gamma$ acts discretely by isometries. In particular, $S$ is a model for $\underline E\Gamma$. However, in general this space has larger dimension than the desired $\vcd(\Gamma)$. In fact, $\dim(S)=\vcd(\Gamma)$ if and only if $\Gamma$ is a uniform lattice. In some rare cases such as $\Gamma=\SL(n,\BZ)$ (see \cite{Ash}) or $\Gamma$ a lattice in a  group of real rank 1 (see Proposition \ref{thin-thick}), the associated symmetric space $S$ admits a $\Gamma$-equivariant cocompact deformation retract $X$ with $\dim(X)=\vcd(\Gamma)$. Such an $X$ is consequently also a cocompact model for $\underline{E}\Gamma$. We do not know if such $\Gamma$-equivariant minimal dimensional cocompact deformation retracts of the symmetric space exist in the generality of the Main Theorem. However, the Main Theorem entails the following weaker statement.

\begin{kor} \label{cor}
If $\Gamma$ is a lattice in a classical simple Lie group $G$, then the symmetric space $S$ of $G$ is $\Gamma$-equivariantly homotopy equivalent to a proper cocompact $\Gamma$-CW complex of dimension $\vcd(\Gamma)$.
\end{kor}
We comment now briefly on the strategy of the proof of the Main Theorem, providing at the same time a section-by-section summary of the paper. Bredon cohomology, introduced by Bredon in \cite{Bredon} to develop an obstruction theory for equivariant extension of maps is an important algebraic tool to study classifying spaces for proper actions. In fact, if $\Gamma$ is a discrete group, then the {\em Bredon cohomological dimension} $\cdm(\Gamma)$ should be thought of as the algebraic counterpart of $\gdim(\Gamma)$. Indeed, it is shown in \cite{LuckMeintrup} that 
$$\cdm(\Gamma)=\gdim(\Gamma)$$
except possibly if $\cdm(\Gamma)=2$ and $\gdim(\Gamma)=3$. It follows in particular that, with the possible exception we just mentioned and which is not going to be relevant in this paper, to prove that $\gdim(\Gamma)=\vcd(\Gamma)$ it suffices to show that $\cdm(\Gamma)=\vcd(\Gamma)$ for lattices $\Gamma$ in classical simple Lie groups $G$. 

Before going any further recall that Borel and Serre \cite{Borel-Serre} constructed a $\Gamma$-invariant bordification $X$ of the symmetric space $S$ associated to $G$; in fact, the {\em Borel-Serre bordification} $X$ is a cocompact model for $\underline E\Gamma$. In section \ref{sec:preli} we recall a few properties of the Borel-Serre bordification, together with the needed terminology about algebraic groups, arithmetic groups and symmetric spaces.

Armed with such a cocompact model $X$ for $\underline E\Gamma$ as the Borel-Serre bordification, we can compute the virtual cohomological dimension of $\Gamma$ as 
$$\vcd(\Gamma)=\max \{  n \in \mathbb{N} \ | \ \mathrm{H}^n_c(X) \neq 0 \}$$
where $\mathrm{H}^n_c(X) $ denotes the compactly supported cohomology of $X$ (see \cite[Cor VIII.7.6]{brown}). In \cite{DMP}, a version of this result was proven for $\cdm(\Gamma)$; we will be using a slight modification of this version.
More concretely, denote by $\CF$ the family of finite subgroups of $\Gamma$ containing the kernel of the action $\Gamma\actson X$, by $X^K$ the fixed point set of some $K\in\CF$, and by $X^K_\mathrm{sing}$ the subcomplex of $X^{K}$ consisting of those cells that are fixed by a finite subgroup of $\Gamma$ that strictly contains $K$. We then have
$$\cdm(\Gamma)= \max\{n \in \mathbb{N} \ | \ \exists K \in \CF \ \mbox{s.t.} \ \mathrm{H}_c^{n}(X^K,X^K_\mathrm{sing})  \neq 0\}$$
In section \ref{sec:tools} we remind the reader of the definition of the Bredon cohomological dimension $\cdm(\Gamma)$ and derive some criteria implying $\vcd(\Gamma)=\cdm(\Gamma)$ for lattices $\Gamma$ in connected simple Lie groups $G$ satisfying appropriate conditions. These conditions are stated in terms of properties of the set $X_{\mathrm{sing}}$ of points in the Borel-Serre bordification $X$ fixed by some non-central element of $\Gamma$.

After these preparatory results we recall what the classical simple Lie groups are in section \ref{sec: classical groups}. For each of these groups $G$ we give a lower bound for the virtual cohomological dimension of a lattice in $G$. This bound is key because the simplest version of the estimates in section \ref{sec:tools} asserts that if $\dim(X_{\mathrm{sing}})<\vcd(\Gamma)$ then $\vcd(\Gamma)=\cdm(\Gamma)$. To estimate the dimensions of fixed point sets we need a few facts on centralizers in compact groups such as $\SO_n$, $\SU_n$ and $\Sp_n$.  We discuss these facts in section \ref{sec:compact}.

Once this is done we can deal with lattices in most classical groups: $\SL(n,\BC)$, $\SO(n,\BC)$, $\Sp(2n,\BC)$, $\SL(n,\BH)$, $\SU(p,q)$, $\Sp(p,q)$, $\Sp(2n,\BR)$ and $\SO^*(2n)$. As we will see in section \ref{sec:easy}, in these cases we have that $\dim(X_{\mathrm{sing}})<\vcd(\Gamma)$ for any lattice $\Gamma$. 

The cases of $\SL(n,\BR)$ and $\SO(p,q)$ are special because they contain lattices with $\dim(X_{\mathrm{sing}})=\vcd(\Gamma)$. In these cases, to obtain that $\cdm(\Gamma)=\vcd(\Gamma)$ we need to prove that the map $H_c^{\vcd(\Gamma)}(X)\to H_c^{\vcd(\Gamma)}(X_{\mathrm{sing}})$ is surjective. We do so by constructing explicit cohomology classes. We treat the case of lattices in $\SL(n,\BR)$ in section \ref{sec:slnr} and that of lattices in $\SO(p,q)$ in section \ref{sec:sopq}. 

The proof of the main theorem and Corollary \ref{cor} will be finalized in Section \ref{sec: proof main}.
\medskip

\noindent \textbf{Acknowledgement.} The first named author was supported by a 2014 Campus Iberus grant, and would like to thank the University of Zaragoza for its hospitality.  The second named author gratefully acknowledges support from the Danish National Research Foundation through the Centre for Symmetry and Deformation (DNRF92) and the third named author  from Gobierno de Arag\'on, European Regional 
Development Funds and
MTM2010-19938-C03-03.

\section{Preliminaries}\label{sec:preli}
In this section we recall some basic facts and definitions concerning algebraic groups, Lie groups, arithmetic groups and symmetric spaces.

\subsection{Algebraic groups}
An \emph{algebraic group} is a subgroup $\BG$ of the special linear group $\SL(N,\BC)$ determined by a collection of polynomials. The algebraic group $\BG$ is defined over a subfield $k$ of $\BC$ if the polynomials in question can be chosen to have coefficients in $k$.  We will be interested in the cases  $k=\BQ$ and $k=\BR$. If $\BG$ is an algebraic group and $R\subseteq\BC$ is a ring we denote by $\BG_R$ the set of those elements of $\BG$ with entries in $R$. It is well-known that for any algebraic group $\BG$ defined over $\BR$ the groups $\BG_\BR$ and $\BG_\BC$ are Lie groups with finitely many connected components. In fact, $\BG$ is (Zariski) connected if and only if $\BG_\BC$ is a connected Lie group. On the other hand, the group of real points of a connected algebraic group defined over $\BR$ need not be connected. A non-abelian algebraic group $\BG$ is {\em simple} if it has no non-trivial, connected normal subgroups. If $\BG$ does not have non-trivial, connected, abelian normal subgroups then $\BG$ is \emph{semisimple}. Note that if $\BG$ is semisimple and defined over $k=\BR$ or $\BC$ then $\BG_k$ is semisimple as a Lie group. The center $Z_{\BG}$ of a semisimple algebraic group $\BG$ is finite and the quotient $\BG/Z_\BG$ is again an algebraic group. Moreover, if the former is defined over $k$ then so is the latter.

An algebraic group $\BT$ is a \emph{torus} if it is isomorphic as an algebraic group to a product $\BC^*\times\dots\times\BC^*$. Equivalently, a connected algebraic group $\BT\subset\SL(N,\BC)$ is a torus if and only if it is diagonalizable, meaning that there is an $A\in\SL(N,\BC)$ such that every element in $A\BT A^{-1}$ is diagonal. If $\BT$ is defined over $k$ then it is said to be \emph{$k$-split} if the conjugating element $A$ can be chosen in $\SL(N,k)$. 

A torus in an algebraic group $\BG$ is an algebraic subgroup that is a torus. A maximal torus of $\BG$ is a torus which is not properly contained in any other other torus. Note that any two maximal tori are conjugate in $\BG$. Also, if $\BG$ is defined over $k$ then any two maximal $k$-split tori are conjugate by an element in $\BG_k$. The $k$-rank of $\BG$, denoted $\rank_k\BG$, is the dimension of some, and hence any, maximal $k$-split torus in $\BG$. 

We refer to \cite{Borel20,Borel-AG} and to \cite{Knapp} for basic facts about algebraic groups and Lie groups respectively.

\subsection{Symmetric spaces}
Continuing with the same notation as above, suppose that $\BG$ is a connected semisimple algebraic group defined over $\BR$ and let $G=\BG_\BR$ be the group of real points. The Lie group $G$ has a maximal compact subgroup $K\subseteq G$.  Recall that any two such maximal compact subgroups are conjugate, and that $K$ is self-normalizing. It follows that we can identify the smooth manifold 
$$S=G/K$$
with the set of all maximal compact subgroups of $G$. Under this identification, left multiplication by $G$ becomes the left action by conjugation on the set of maximal compact subgroups. 

Note that $S$ is contractible because the inclusion of $K$ into $G$ is a homotopy equivalence. In fact, as $S$ admits a $G$-invariant symmetric Riemannian metric of non-positive curvature, it follows that $S$ is diffeomorphic to Euclidean space. From now on we will consider $S$ to be endowed with such a symmetric metric. The space $S$ is called the \emph{symmetric space} of $G$.

Recall that two Lie groups are isogenous if they are locally isomorphic, meaning that their identity components have isomorphic universal covers. Observe that isogenous groups have isometric associated symmetric spaces. 

A {\em flat} $F$ in the symmetric space $S$ is a totally geodesic submanifold isometric to $\BR^d$ for some $d$. A {\em maximal flat} is a flat which is not properly contained in any other flat. It is well-known that the identity component $G^{0}$ of the Lie group $G$ acts transitively on the set of maximal flats. 

Flats and tori are intimately linked to each other: if $\BT\subseteq\BG$ is a $\BR$-split torus with group of real points $T=\BT_\BR$, then the action of $T$ on the symmetric space $S$ leaves invariant a flat $F_T$ of dimension $\dim F_T=\dim T$. Actually, this yields a 1-to-1 correspondence between maximal flats in $F$ and maximal $\BR$-split tori in $G$, and we have the following metric characterization of the real rank: {\em If $\BG$ is a semisimple algebraic group defined over $\BR$, $G$ is the group of real points and $K\subseteq G$ is a maximal compact subgroup, then $\rank_\BR\BG$ is equal to the maximal dimension of a flat in $G/K$.} Note that this implies that if $\rank_\BR\BG\ge 2$ then the sectional curvature of $S$ cannot be negative. In fact, it is well-known that both statements are equivalent: {\em $S$ is negatively curved if and only if $\rank_\BR(\BG)=1$}. 

Most of our work will rely on computations of dimensions of fixed point sets 
$$S^H=\{x\in S \ \vert \ hx=x\ \hbox{for all}\ h\in H\}$$ 
of finite subgroups $H\subset K$. There is an identification 
$$C_G(H)/C_K(H)\to S^H,\ \ gC_K(H)\mapsto gK$$
between the symmetric space $C_G(H)/C_K(G)$ and  the fixed point set $S^H$, where $C_K(H)$ and $C_G(H)$ are the centralizers of $H$ in $K$ and $G$ respectively. We state here for further reference the following consequence of this fact.

\begin{prop}\label{cartan}
Let $\BG$ be a semisimple algebraic group defined over $\BR$, $G$ the group of real points, $K\subset G$ a maximal compact subgroup and $S=G/K$ the associated symmetric space. For any subgroup $H\subset K$ we have
$$\dim S^H=\dim C_G(H)-\dim C_K(H)$$
where $C_K(H)$ and $C_G(H)$ are the centralizers of $H$ in $K$ and $G$ respectively.
\end{prop}

With the same notation as in Proposition \ref{cartan} we note that in fact $C_K(H)$ is a maximal compact subgroup of $C_G(H)$. This implies the following.

\begin{kor}\label{complexification}
Let $G=\BG_\BC$ be the group of complex points of a semisimple algebraic group $\BG$, $K\subset G$ be maximal compact, and $S=G/K$. For any subgroup $H\subset K$ we have $\dim S^H=\dim C_K(H)$.
\end{kor}
\begin{proof}
The group $C_G(H)$ is the group of complex points of a reductive algebraic group. In particular, it is the complexification of its maximal compact subgroup $C_K(H)$, so the dimension of $C_G(H)$ is twice that of $C_K(H)$. 
\end{proof}

We refer to \cite{Mostow,Helgason,Eberlein} for facts about symmetric spaces and to \cite{BGS} for the geometry of manifolds of non-positive curvature.

\subsection{Arithmetic groups}
Recall that a connected semisimple algebraic group $\BG$ is defined over $\BQ$ if and only if the closure of the group $\BG_\BQ$ of rational points contains the identity component $\BG_\BR^0$ of the Lie group $\BG_\BR$ of real points \cite[Proposition 5.8]{Witte}. For any such $\BG$, the group $\BG_\BZ$ is a discrete subgroup of $\BG_\BR$. In fact, $\BG_\BZ$ is a lattice, meaning that $\BG_\BZ\bs\BG_\BR$ has finite Haar measure. A lattice is \emph{uniform} if it is cocompact and \emph{non-uniform} otherwise. 

Although the group $\BG_\BZ$ is the paradigm of an arithmetic group, the definition is slightly more general. Since we are going to be mostly concerned with non-uniform lattices in simple Lie groups, we shall only define arithmetic groups in that restrictive setting. Let $G$ be a non-compact simple Lie group and $\Gamma\subset G$ a non-uniform lattice. The lattice $\Gamma$ is \emph{arithmetic} if there is a connected algebraic group $\BG$ defined over $\BQ$, a finite normal subgroup $Z\subset G$ and a Lie group isomorphism 
$$\phi:G/Z\to\BG_\BR^0$$
where $\BG_\BR^0$ is the identity component of the group of real points of $\BG$ such that $\phi(\Gamma)$ is commensurable to $\BG_\BZ$. After referring the reader to for instance \cite{Witte} for a definition of arithmetic groups in all generality, we note that by moding out the center we may assume without loss of generality that the algebraic group $\BG$ is center-free. This implies that the commensurator of $\BG_\BZ$ in $\BG$ is $\BG_\BQ$ and hence that $\phi(\Gamma)\subset\BG_\BQ$. 

The content of the Margulis's arithmeticity theorem is that under certain mild assumptions, every lattice is arithmetic. We state it only in the restricted setting we will be working in.

\begin{named}{Arithmeticity theorem}[Margulis]
Let $G$ be the group of real points of a simple algebraic group defined over $\BR$. If $G$ is not isogenous to $\SO(1,n)$ or to $\SU(1,n)$, then every lattice in $G$ is arithmetic.
\end{named}

Observe also that both $\SO(1,n)$ and $\SU(1,n)$ have real rank 1. In particular, the arithmeticity theorem applies to every lattice in a group with $\rank_\BR\ge 2$. 

Before moving on, we note that with the same notation as above, both groups $\Gamma$ and $\phi(\Gamma)$ have the same virtual cohomological and proper geometric dimensions. This follows from the following general fact:

\begin{lem}\label{lem:inv-finite extensions}
Suppose  $\Gamma$ is a virtually torsion-free subgroup and  $F\subset\Gamma$ is a normal finite subgroup. Then $\vcd(\Gamma)=\vcd(\Gamma/F)$ and $\gdim(\Gamma)=\gdim(\Gamma/F)$. 
\end{lem}
\begin{proof}
Both groups have the same virtual cohomological dimension because they contain isomorphic torsion-free finite index subgroups. To prove the second claim suppose that $X= \underline{E}\Gamma$ is a model for $\Gamma$. The fixed point set $X^F$ admits a $\Gamma/F$-action  and is in fact a model for $\underline{E}(\Gamma/F)$. On the other hand, if $Y=\underline{E}(\Gamma/F)$ is a model for $\Gamma$, then it is also a model for $\Gamma$ via the action induced by the projection $\Gamma\to\Gamma/F$.
\end{proof}

We refer to \cite{Witte,Margulis,Ji-li,Ji-MacPherson} and mostly to the beautiful book \cite{Borel-arith} for facts and definitions on arithmetic groups.

\subsection{Rational flats}
Continuing with the same notation as in the preceding paragraphs, let $G$ be the Lie group of real points of a connected semisimple algebraic group $\BG$ defined over $\BQ$,  $K\subset G$  a maximal compact subgroup and  $S=G/K$ the associated symmetric space. A flat $F$ in $S$ is {\em rational} if the map $F\to \BG_\BZ\bs S$ is proper. A maximal rational flat is then a rational flat which is not contained in any other rational flat. Maximal rational flats arise as orbits of maximal $\BQ$-split tori and in fact we have the following geometric interpretation of the rational rank: {\em If $\BG$ is a semisimple algebraic group defined over $\BQ$, $G=\BG_\BR$ is the group of real points, $K\subset G$ is a maximal compact subgroup and $S=G/K$ is the associated symmetric space, then $\rank_\BQ\BG$ is equal to the dimension of a maximal rational flat.}

Note that $\rank_\BQ\BG=0$ if and only if $\BG_\BZ\bs S$ is compact, meaning that $\BG_\BZ$ has no non-trivial unipotent element. Similarly, $\rank_\BQ\BG=1$ if and only if non-trivial unipotent elements in $\BG_\BZ$ are contained in unique maximal unipotent subgroups. For higher $\BQ$-rank there is no such simple characterization but one has however the following useful fact.

\begin{prop}\cite[Prop. 9.15]{Witte}\label{prop-rank2}
Assume that $G=\BG_\BR$ is the group of  real points of a connected simple algebraic group defined over $\BQ$. If $\rank_\BQ\BG\ge 2$, then $\BG_\BZ$ contains a subgroup commensurable to either $\SL(3,\BZ)=\SL(3,\BC)_\BZ$ or to $\SO(2,3)_\BZ$.
\end{prop}

We note at this point that $\SO(2,3)$ is isogenous to $\Sp(4,\BR)$.
\medskip

Before moving on we add a comment which will come handy later on. Suppose  that $\BG$ is a semisimple algebraic group defined over $\BQ$, let $G=\BG_\BR$ be the group of real points and $S=G/K$ the corresponding symmetric space. Recall that $\BG_\BQ$ is dense in the identity component $G^0$ of $G$ and that $G^0$ acts transitively on the set of maximal flats. It follows that the $\BG_\BQ$-orbit of any maximal flat is dense in the set of all maximal flats. In particular we get that if there is a maximal flat which is rational, then the set of all rational maximal flats is dense in the set of all maximal flats. We record this fact in the following lemma.

\begin{lem}\label{density flats}
Let $\BG$ be a semisimple algebraic group defined over $\BQ$, let $G=\BG_\BR$ be the group of real points, $K\subset G$ a maximal compact subgroup, $S=G/K$ the corresponding symmetric space and $F\subset S$ a maximal flat. If $\rank_\BQ(\BG)=\rank_\BR(\BG)$ and if $U\subset G^0$ is open and non-empty, then there is $g\in U$ such that $gF$ is a  maximal  rational flat.\qed
\end{lem}

\subsection{Models for $\underline E\Gamma$ if $\rank_\BR\BG=1$}
Continuing with the same notation as above suppose that $G=\BG_\BR$ is the group of real points of a semisimple algebraic group $\BG$, $K\subset G$ is a maximal compact subgroup, $S=G/K$ is the associated symmetric space and $\Gamma\subset G$ is a lattice. Being a simply connected complete manifold of non-positive curvature, the symmetric space $S$ is a CAT(0)-space. This implies that $S$ is a model for $\underline E\Gamma$ that, in general, is not cocompact. Our next aim is to explain the existence of cocompact models. Assume that $\rank_\BR(\BG)=1$ and recall that this means that the symmetric space $S$ is negatively curved. Now, it follows from the Margulis Lemma, or equivalently from the existence of the thin-thick decomposition, that if $\Gamma\subset G$ is a lattice then there is a $\Gamma$-invariant set $\CB\subset S$ whose connected components are open horoballs and such that $X=S\setminus\CB$ is a smooth manifold with boundary on which $\Gamma$ acts cocompactly. Moreover, the convexity of the distance function on $S$  easily implies that $X$ is also a model for $\underline E\Gamma$. At this point we note that one can extract from $X$ an even smaller model for $\underline E\Gamma$, namely the `cut locus from the boundary', i.e. the set $\hat X$ of those points $x\in X$ for which there are at least 2 minimizing geodesics to $\D X$. This is an analytic subset with empty interior to which $X$ retracts in a $\Gamma$-equivariant way. In other words, $\hat X$ is a model for $\underline E\Gamma$ of one dimension less than $X$. Summarizing, we have the following well-known result.

\begin{prop}\label{thin-thick}
Let $\BG$ be an algebraic group defined over $\BR$ with $\rank_\BR\BG=1$, let $G\subset\BG$ be the group of real points, $K\subset G$ a maximal compact subgroup and $S=G/K$ the associated symmetric space. Suppose that $\Gamma\subset G$ is lattice. If $\Gamma$ is not cocompact, then there are cocompact models $X,\hat X$ for $\underline E\Gamma$ such that
\begin{itemize}
\item[-] $X$ is a manifold of dimension $\dim X=\dim S$, and
\item[-] $\hat X$ is a CW-complex of dimension $\dim\hat X=\dim S-1$.
\end{itemize}
\end{prop}

\begin{bem}
Note that Proposition \ref{thin-thick} admits the following generalization: {\em Suppose that $\Gamma\actson M$ is a smooth isometric action on a complete Riemannian manifold such that $M$ is an $\underline E\Gamma$. If $\Gamma\bs M$ is not compact, then $\gdim(\Gamma)\le\dim M-1$.} Indeed, since $\Gamma \setminus M$ is not compact one can find a $\Gamma$-invariant set $Z\subset M$ such that
\begin{itemize}
\item[(1)] each component of $Z$ is a geodesic ray,
\item[(2)] the distance between points in any two components is at least $1$,
\item[(3)] the set $Z$ is maximal in the sense that there is no minimizing geodesic ray $\gamma$ such that $t\mapsto d(Z,\gamma(t))$ tends to $\infty$. 
\end{itemize}
By removing an appropriate $\Gamma$-invariant regular neighborhood of $Z$ from $M$, one obtains a manifold $X$ with boundary of the same dimension as $M$ that is still a  model for $\underline{E}\Gamma$. Then we proceed as above by taking the cut-locus to the boundary in $X$ to obtain a model $\hat{X}$ for $\underline{E}\Gamma$ of one dimension less.
\end{bem}

Continuing with the same notation as above Proposition \ref{thin-thick}, observe that $\D X$ is either empty or the disjoint union of infinitely many horospheres. Since horospheres are contractible we conclude from Poincar\'{e}-Lefschetz duality that $H_c^k(X)=0$ unless $k=\dim X$ and $\D X=\emptyset$, or $k=\dim X-1$ and $\D X\neq\emptyset$. Taking into account that 
$$\vcd(\Gamma)=\max \{  n \in \mathbb{N} \ | \ \mathrm{H}^n_c(X) \neq 0 \}$$
we obtain the following.

\begin{prop}\label{borel-serre-vcd1}
Let $\BG$ be an algebraic group defined over $\BR$ with $\rank_\BR\BG=1$, let $G\subset\BG$ be the group of real points, $K\subset G$ a maximal compact subgroup and $S=G/K$ the associated symmetric space. Suppose that $\Gamma\subset G$ is a lattice. If $\Gamma$ is cocompact, then $\vcd\Gamma=\dim X$. Otherwise we have $\vcd\Gamma=\dim X-1$.
\end{prop}

Combining Proposition \ref{thin-thick} and Proposition \ref{borel-serre-vcd1}, and recalling that $\vcd(\Gamma)$ is always a lower bound for $\gdim(\Gamma)$, we get the following.

\begin{kor}\label{crit0}
Let $\BG$ be an algebraic group defined over $\BR$ and let $\Gamma\subset\BG_\BR$ be a lattice in the group of real points of $\BG$. If $\rank_\BR\BG=1$ then $\gdim(\Gamma)=\vcd(\Gamma)$.\qed
\end{kor}

We refer to \cite{BGS} for a discussion of the Margulis Lemma and its consequences. 

\subsection{Models for $\underline E\Gamma$ if $\rank_\BR\BG\ge 2$}\label{subsec-borel-serre}
Suppose now that $G$ is a simple Lie group with $\rank_\BR\BG\ge 2$ and let $\Gamma\subset G$ be a non-uniform lattice, which is arithmetic by Margulis's theorem. Recall that this means that there is a connected algebraic group $\BG$ defined over $\BQ$, a finite normal subgroup $Z\subset G$ and a Lie group isomorphism 
$$\phi:G/Z\to\BG_\BR^0$$
where $\BG_\BR^0$ is the identity component of the group of real points of $\BG$ such that $\phi(\Gamma)\subset\BG_\BQ$ is commensurable to $\BG_\BZ$. Suppose that $K$ is a maximal compact subgroup of $G$ and note that also $\phi(K)$ is a maximal compact subgroup of $\BG_\BR^0$. In particular, the symmetric spaces $S=G/K$ and $\BG_\BR^0/\phi(K)$ are identical. 

Borel and Serre \cite{Borel-Serre} constructed an analogue of the manifold $X$ above called {\em Borel-Serre bordification of $S$ associated to the rational structure of $\BG$}, or just simpler {\em Borel-Serre bordification of $S$} (see also \cite{Grayson,Ji-MacPherson,Leuzinger-BS}). The manifold $X$ is a bordification of the symmetric space $S$, invariant under $\BG_\BQ$, and on which $\BG_\BZ$, and hence $\phi(\Gamma)$ acts cocompactly. Moreover, since the construction of $X$ behaves well with respect to subgroups one has that $X$ is a cocompact model for $\underline E\Gamma$. 

\begin{sat}[Borel-Serre]\label{borel-serre-thm}
Let $G$ be a simple Lie group with $\rank_\BR\BG\ge 2$, $K\subset G$ a maximal compact subgroup, $S=G/K$ the associated symmetric space and $\Gamma\subset G$ a non-uniform lattice.

The Borel-Serre bordification $X$ of $S$ is a cocompact model for $\underline E\Gamma$. Moreover, if $H\subset\Gamma$ is a finite group then the fixed point set $X^H$ is $C_\Gamma(H)$-equivariantly homeomorphic to the Borel-Serre bordification of $S^H$, meaning in particular that $X^H$ is a model for for $\underline E(C_\Gamma(H))$. 
\end{sat}

Recall that each boundary component of the manifold provided by Proposition \ref{thin-thick} is contractible. In general, this is not true for the Borel-Serre bordification $X$. In fact, $\D X$ is homotopy equivalent to the spherical building associated to the group $\BG_\BQ$ and hence to a wedge $\bigvee_{i \in I} S^{\rank_\BQ(\BG)-1}$ of spheres of dimension $\rank_\BQ(\BG)-1$. This implies that 
\begin{equation}\label{equation-raining-a-lot}
\mathrm{H}_{\rank_\BQ(\BG)}(X,\partial X)\cong \bigoplus_{i \in I}\mathbb{Z}
\end{equation}
and hence the following (see \cite{Borel-Serre}).

\begin{sat}[Borel-Serre]\label{borel-serre-vcd2}
Let $G$ be a simple Lie group with $\rank_\BR G\ge 2$, $K\subset G$ a maximal compact subgroup, $S=G/K$ the associated symmetric space and $\Gamma\subset G$ a non-uniform lattice. If $\BG$ is a simple algebraic group defined over $\BQ$ such that there is a finite group $Z\subset G$ and an isomorphism $\phi:G/Z\to\BG_\BR^0$ with $\phi(\Gamma)$  commensurable to $\BG_\BZ$, then $\vcd\Gamma=\dim S-\rank_\BQ\BG$.
\end{sat}

Before moving on recall that the $\BQ$-rank of $\BG$ is bounded  above by the $\BR$-rank. In particular we get the following lemma from Theorem \ref{borel-serre-vcd2}.

\begin{lem}\label{borel-serre}
Let $G$ be a simple Lie group with $\rank_\BR G \ge 2$, $K\subset G$ a maximal compact subgroup, $S=G/K$ the associated symmetric space and $\Gamma\subset G$ a non-uniform lattice. If $\BG$ is a simple algebraic group defined over $\BQ$ such that there are a finite group $Z\subset G$ and an isomorphism $\phi:G/Z\to\BG_\BR^0$ with $\phi(\Gamma)$  commensurable to $\BG_\BZ$, then 
$$\vcd(\Gamma)\ge\dim G/K-\rank_\BR G$$
with equality if and only of $\rank_\BQ(\BG)=\rank_\BR(\BG)$.\qed
\end{lem}

The appeal of Lemma \ref{borel-serre} is that the rational rank of a group defined over $\BQ$ depends on something as complicated as its isomorphism type over $\BQ$ while on the other hand, the real rank only depends on the type over $\BR$, something which is much simpler to deal with. 
\medskip

Before moving on we wish to add a comment related to the isomorphism \eqref{equation-raining-a-lot}. Sticking to the notation above, suppose that $F\simeq\BR^{\rank_\BQ(\BG)}$ is a maximal rational flat and let $\D F\subset \partial X$ be its sphere at infinity, which is homologically essential in $\D X$. We conclude the following.

\begin{prop}\label{prop: flat cohom class}
Let $\BG$ be a simple algebraic group defined over $\BQ$, let $G\subset\BG$ be the group of real points, $K\subset G$ a maximal compact subgroup, $S=G/K$ the associated symmetric space, $X$ the Borel-Serre bordification of $S$ and $F$ a maximal rational flat in $S$. The inclusion of $F$ into $S$ extends to map 
$$(F,\partial F)\to(X,\partial X)$$
such that the induced map on homology
\[   \mathbb{Z}  \cong \mathrm{H}_{\rank_\BQ(\BG)}(F,\partial F)\rightarrow  \mathrm{H}_{\rank_\BQ(\BG)}(X,\partial X)\]
takes a generator of  $\mathrm{H}_{\rank_\BQ(\BG)}(F,\partial F)$ to a generator of an infinite cyclic summand of $\mathrm{H}_{\rank_\BQ(\BG)}(X,\partial X)$.
\end{prop}

%
%
%
%

\section{Cohomological tools}\label{sec:tools}

In this section we recall the definition of the Bredon cohomological dimension $\cdm(\Gamma)$ of a discrete group $\Gamma$ and derive some criteria implying $\vcd(\Gamma)=\cdm(\Gamma)$ for certain lattices $\Gamma$ in simple Lie groups.
\medskip

Let $\Gamma$ be a discrete group and let $\CF$ be a family of subgroups of $\Gamma$ that is closed under conjugation and intersections. The \emph{orbit category} $\mathcal{O}_{\CF}\Gamma$ is the category whose objects  are left coset spaces $\Gamma/H$ with $H \in \CF$ and where the morphisms are all $\Gamma$-equivariant maps between the objects. An \emph{$\mathcal{O}_{\CF}\Gamma$-module} is a contravariant functor \[M: \mathcal{O}_{\CF}\Gamma \rightarrow \mathbb{Z}\mbox{-mod}.\] The \emph{category of $\mathcal{O}_{\CF}\Gamma$-modules}, denoted by $\mbox{Mod-}\mathcal{O}_{\CF}\Gamma$, has all the $\mbox{Mod-}\mathcal{O}_{\CF}\Gamma$-modules as objects and all the natural transformations between these objects as morphisms. One can show that $\mbox{Mod-}\mathcal{O}_{\CF}\Gamma$ is an abelian category that contains enough projective modules to construct projective resolutions. Hence, one can construct functors $\mathrm{Ext}^{n}_{\mathcal{O}_{\CF}\Gamma}(-,M)$ that have all the usual properties. The \emph{$n$-th Bredon cohomology of $\Gamma$} with coefficients $M \in \mbox{Mod-}\mathcal{O}_{\CF}\Gamma$ is by definition
\[ \mathrm{H}^n_{\CF}(\Gamma,M)= \mathrm{Ext}^{n}_{\mathcal{O}_{\CF}\Gamma}(\underline{\mathbb{Z}},M), \]
where $\underline{\mathbb{Z}}$ is the functor that maps all objects to $\mathbb{Z}$ and all morphisms to the identity map. For more details, we refer the reader to \cite[Section 9]{Luck}.

The \emph{Bredon cohomological dimension of $\Gamma$ for proper actions}, denoted by $\cdm(\Gamma)$ is defined as
\[ \cdm(\Gamma) = \sup\{ n \in \mathbb{N} \ | \ \exists M \in \mbox{Mod-}\mathcal{O}_{\mathcal{FIN}}\Gamma :  \mathrm{H}^n_{\mathcal{FIN}}(\Gamma,M)\neq 0 \}, \]
where $\mathcal{FIN}$ is the family of all finite subgroups of $\Gamma$. As mentioned in the introduction, the invariant $\cdm(\Gamma)$ should the viewed as the algebraic counterpart of $\gdim(\Gamma)$. Indeed, it is known that these two notions of dimension coincide (see \cite[0.1]{LuckMeintrup}), except for the possibility that one could have $\cdm(\Gamma)=2$ but $\gdim(\Gamma)=3$ (see \cite{BradyLearyNucinkis}).

\begin{sat}[L\"uck-Meintrup]\label{Luck-Meintrup}
If $\Gamma$ be a discrete group with $\cdm(\Gamma)\ge 3$, then $\gdim(\Gamma)=\cdm(\Gamma)$.
\end{sat}

As we mentioned in the introduction, if $\Gamma$ is virtually torsion-free, then $\vcd(\Gamma)\leq \cdm(\Gamma)$ but equality does not hold in general. Recall also that if $X$ is any cocompact model for $\underline E\Gamma$, then $\vcd(\Gamma)=\max \{  n \in \mathbb{N} \ | \ \mathrm{H}^n_c(X) \neq 0 \}$, where $\mathrm{H}^n_c(X) $ denotes the compactly supported cohomology of $X$ (see \cite[Cor. 7.6]{brown}). There is an analogue of this result for $ \cdm(\Gamma)$.

\begin{sat}{\rm \cite[Th. 1.1]{DMP}}\label{th: bredon compact support} 
Let $\Gamma$ be a group that admits a cocompact model $X$ for $\underline{E}\Gamma$. Then
\[    \cdm(\Gamma)= \max\{n \in \mathbb{N} \ | \ \exists K \in \CF \ \mbox{s.t.} \ \mathrm{H}_c^{n}(X^K,X^K_\mathrm{sing})  \neq 0\}   \] 
where $\CF$ is the family of finite subgroups of $\Gamma$ containing the kernel of the action $\Gamma\actson X$, and where $X^K_\mathrm{sing}$ is the subcomplex of $X^{K}$ consisting of those cells that are fixed by a finite subgroup of $G$ that strictly contains $K$.
\end{sat}

\begin{bem}
In \cite{DMP}, this theorem is stated in such a way that $\CF$ is the family of all finite subgroups of $\Gamma$. Both formulations are easily seen to be equivalent  (compare with Lemma \ref{lem:inv-finite extensions}).
\end{bem}

Before moving on we introduce some notation which we will use throughout the paper. If $X$ is a model for $\underline E\Gamma$ then $X_{\mathrm{sing}}$ is the subspace of $X$ consisting of points whose stabilizer is strictly larger than the kernel of the action $\Gamma\actson X$.
\medskip

After these general reminders we focus in the concrete situation we are interested in, namely that of lattices in simple Lie groups of higher rank. Recall that such lattices are arithmetic by Margulis's theorem. Notation will be as in section \ref{subsec-borel-serre}:

\begin{itemize}
\item $G=\BG_\BR$ is the group of real points of a connected simple algebraic group with $\rank_\BR\BG\ge 2$.
\item $\Gamma\subset G$ is a non-uniform lattice.
\item $K\subset G$ is a maximal compact subgroup.
\item $S=G/K$ is the associated symmetric space.
\item $X$ is the Borel-Serre bordification of $S$.
\end{itemize}
Recall that the Borel-Serre bordification $X$ of $S$ is a model for $\underline E\Gamma$ by Theorem \ref{borel-serre-thm} and that we have $\dim(S^H)=\dim(X^H)$ for any finite subgroup $H$ of $\Gamma$. Moreover, the kernel of the action $\Gamma\actson X$ is exactly the center of $\Gamma$ so $S^H=S$ if and only if $H$ is central. This implies that, with the notation introduced above, $\CF$ is the family of finite subgroups of $\Gamma$ lying over the center and $X_{\mathrm{sing}}$ is the subspace of $X$ consisting of points that are fixed by a non-central finite order element of $\Gamma$.  We denote by 
$$\CS=\{X^H \ \vert\  H\in\CF\text{ non-central and } \nexists H'\in\CF\ \hbox{ non-central with}\ X^H\subsetneq X^{H'}\}$$ 
the set of fixed point sets of non central elements $H\in\CF$ which are maximal. Note that 
$$X_{\mathrm{sing}}=\bigcup_{X^H\in\CS}X^H$$
Also, every fixed point set of $\CS$ is actually of the form $X^A$, where $A$ is a non-central finite order group element of $\Gamma$. We are now ready to prove the first criterion ensuring that $\vcd(\Gamma)=\cdm(\Gamma)$.

\begin{prop}\label{prop-coho1}  
If 
\begin{itemize}
\item[(1)] $\dim(X^A)\leq \vcd(\Gamma)$ for every $X^A\in\CS$, and
\item[(2)] the homomorphism $\mathrm{H}_c^{\vcd(\Gamma)}(X)\to \mathrm{H}_c^{\vcd(\Gamma)}(X_{\sing})$ is surjective
\end{itemize}
then $\vcd(\Gamma)=  \cdm(\Gamma)$.
\end{prop}

\begin{proof}
First recall that in general one has $\vcd(\Gamma)\leq \cdm(\Gamma)$. The long exact sequence for the pair $(X,X_{\mathrm{sing}})$ and the fact that $\mathrm{H}_c^n(X_{\sing})$ equals zero  for $n>\vcd(\Gamma)$ because of (1), imply that $\mathrm{H}_c^{n+1}(X,X_{\mathrm{sing}})=0$ for all $n>\vcd(\Gamma)$ and entails the exact sequence
$$\mathrm{H}_c^{\vcd(\Gamma)}(X)\to\mathrm{H}_c^{\vcd(\Gamma)}(X_{\sing}) \rightarrow  \mathrm{H}_c^{\vcd(\Gamma)+1}(X,X_{\mathrm{sing}})\rightarrow 0.$$
Since $\mathrm{H}_c^{\vcd(\Gamma)}(X)\to\mathrm{H}_c^{\vcd(\Gamma)}(X_{\sing})$ is surjective by assumption we get that 
$$\mathrm{H}_c^{n}(X,X_{\mathrm{sing}})=0$$
for all $n>\vcd(\Gamma)$. Since $\dim(X^H)\leq \vcd(\Gamma)$ for all non-central $H\in \CF$, it follows from Theorem \ref{th: bredon compact support} that $\cdm(\Gamma)\le\vcd(\Gamma)$, proving the proposition.
\end{proof}
Since $\dim(S^H)=\dim(X^H)$ for any finite subgroup $H$ of $\Gamma$,  the following corollary is immediate.

\begin{kor}\label{crit1}
If $\dim S^A < \vcd(\Gamma)$ for every non-central finite order  element $A \in \Gamma$, then $\vcd(\Gamma)=  \cdm(\Gamma)$.\qed
\end{kor}

The simple criterion in Corollary \ref{crit1} is actually going to deal with a lot of the cases in the main result, but not with all. In certain situations we will have $\dim X^A = \vcd(\Gamma)$ for certain $X^A \in \CS$, in which case we will have to prove that condition (2) of Proposition \ref{prop-coho1} is satisfied. The following proposition will help us do that.

\begin{prop}\label{mayer vietoris} Fix an integer $d\geq 2$. If 
$$\dim(X^A\cap X^B)\le d-2$$ 
for any two distinct $X^A, X^B\in\CS$, then the homomorphism
$$\mathrm{H}_c^d(X)\to \mathrm{H}_c^d(X_{\sing})$$
is surjective if for any finite subset $\{X^{A_1},\ldots,X^{A_n}\}$ of $\mathcal{S}$, the homomorphism
$$\mathrm{H}_c^d(X)\to \mathrm{H}_c^d(X^{A_1})\oplus\dots\oplus \mathrm{H}_c^d(X^{A_n})$$
is surjective.

\end{prop}

Proposition \ref{mayer vietoris} will follow easily from the following observation.

\begin{lem}\label{covering} Let $Y$ be a $CW$-complex with a covering $Y=\cup_{\alpha\in\Lambda}Y_\alpha$ by subcomplexes. Let $Z\subseteq Y$ be a subcomplex and assume that whenever $\alpha\neq\beta\in\Lambda$, $Y_\alpha\cap Y_\beta\subseteq Z$. Then there is an isomorphism of compactly supported cochain complexes
$$\mathcal C^*_c(Y,Z)\buildrel{\cong}\over\to\bigoplus_{\alpha\in\Lambda}\mathcal C^*_c(Y_\alpha,Z\cap Y_\alpha).$$
\end{lem}
\begin{proof} As each $Y_\alpha$ is a subcomplex of $Y$ we have restriction maps \[\mathcal C^*_c(Y,Z)\to\mathcal C^*_c(Y_\alpha,Z\cap Y_\alpha).\] For each $f\in\mathcal C^*_c(Y,Z)$, any $\sigma$ in the support of $f$ belongs to just one of the sets $Y_\alpha$ and since we are working with compactly supported cochains, we deduce that $f$ vanishes when restricted to almost all $Y_\alpha$'s, so we get \[\phi:\mathcal C^*_c(Y,Z)\to\bigoplus_{\alpha\in\Lambda}\mathcal C^*_c(Y_\alpha,Z\cap Y_\alpha).\]  We will construct an inverse to this map.
Let
$$\oplus_{\alpha\in\Lambda}f_\alpha\in\bigoplus_{\alpha\in\Lambda} C^*_c(Y_\alpha,Z\cap Y_\alpha)$$ be a sum of compactly supported cochain maps and define
$$\psi(\oplus_{\alpha\in\Lambda}f_\alpha)(\sigma)=\Bigg\{\begin{aligned}
&f_\alpha(\sigma),\text{ if there is a unique $\alpha\in\Lambda$ so that }\sigma\in Y_\alpha\\
&0\text{ \ \ otherwise}.\end{aligned}$$

It is not difficult to verify that this gives a well-defined chain map $\psi$, and that $\psi$ and $\phi$ are inverses of each other.
\end{proof}

\begin{proof}[Proof of Proposition \ref{mayer vietoris}]
Letting  \[Z=\bigcup_{X^A\neq X^B \in \mathcal{S}}(X^A\cap X^B),\] the chain isomorphism of Lemma \ref{covering} implies that
$$\mathrm{H}^*_c(X_{\sing},Z)=\bigoplus_{X^A\in\CS}\mathrm{H}^*_c(X^A,Z\cap X^A).$$
Now, the long exact sequence  
$$\ldots\to \mathrm{H}_c^{d-1}(Z)\to \mathrm{H}_c^d(X_{\sing},Z)\to \mathrm{H}_c^d(X_{\sing})\to \mathrm{H}_c^{d}(Z)\to\ldots$$ of the pair $(X_{\sing},Z)$, together with the fact that $\dim Z\le d-2$, imply that $\mathrm{H}_c^d(X_{\sing},Z)= \mathrm{H}^d_c(X_{\sing})$. By exactly the same reason we have $\mathrm{H}_c^d(X^A,Z\cap X^A)= \mathrm{H}^d_c(X^{A})$ for any $X^A\in\CS$. We conclude that the inclusions $X^A \rightarrow X$, for $X^A \in \mathcal{S}$, induce an isomorphism
 $$\mathrm{H}^d_c(X_{\sing})\xrightarrow{\cong }\bigoplus_{X^A\in\CS}\mathrm{H}^d_c(X^A)$$
which implies the proposition.
\end{proof}

In almost all cases of interest where Corollary \ref{crit1} does not apply, we will be able to use the following corollary.

\begin{kor} \label{cor: cohom flat}
With the same notation as above, suppose that
\begin{enumerate}
\item $\dim(S^A)\leq \vcd(\Gamma)$ for every non-central finite order element $A\in\Gamma$,
\item $\dim(S^A\cap S^B)\le \vcd(\Gamma)-2$ for any two distinct  $S^A, S^B\in\mathcal{S}$, and 
\item for any finite set of non-central finite order elements $A_1,\dots,A_r\in\Gamma$ with $S^{A_i}\neq S^{A_j}$ for $i\neq j$,  $\dim(S^{A_i})=\vcd(\Gamma)$, and such that $C_{\Gamma}(A_i)$ is a cocompact lattice in $C_G(A_i)$, there exists a rational flat $F$ in $S$ that intersects $S^{A_1}$ in exactly one point and is disjoint from $S^{A_i}$ for $i \in \{2,\ldots,r\}$.
\end{enumerate}
Then $\vcd(\Gamma)=  \cdm(\Gamma)$.
\end{kor}
\begin{proof} By Propositions \ref{prop-coho1} and \ref{mayer vietoris}, it suffices to show that 
for any finite set of pairwise distinct elements $X^{A_1},\dots,X^{A_r}\in \mathcal{S}$, the map
$$\mathrm{H}_c^{\vcd(\Gamma)}(X)\to \mathrm{H}_c^{\vcd(\Gamma)}(X^{A_1})\oplus\dots\oplus \mathrm{H}_c^{\vcd(\Gamma)}(X^{A_r})$$
is surjective. Since $\vcd(C_{\Gamma}(A_i))=\dim(S^{A_i})-\rank_\BQ(C_\Gamma(A_i))\leq \vcd(\Gamma)$ and $\dim(S^{A_i})\leq \vcd(\Gamma)$, the fact that  $\mathrm{H}_c^{\vcd(\Gamma)}(X^{A_1})\neq 0$ implies that  $\dim(X^{A_i})=\vcd(\Gamma)$  and $\rank_\BQ(C_\Gamma(A_i))=0$; in particular, $C_{\Gamma}(A_i)$ is a cocompact lattice in $C_G(A_i)$. Therefore it suffices to show that for any finite set of pairwise distinct elements $X^{A_1},\dots,X^{A_r}\in \mathcal{S}$ such that $\dim(X^{A_i})=\vcd(\Gamma)$ and $C_{\Gamma}(A_i)$ is a cocompact lattice in $C_G(A_i)$,  the map
\begin{equation}\label{eq: map}
\mathrm{H}_c^{\vcd(\Gamma)}(X)\to \mathrm{H}_c^{\vcd(\Gamma)}(X^{A_1})\oplus\dots\oplus \mathrm{H}_c^{\vcd(\Gamma)}(X^{A_r})
\end{equation}
is surjective. Let $X^{A_1},\dots,X^{A_r}\in \CS$ be such a collection of elements and note that $\mathrm{H}_c^{\vcd(\Gamma)}(X^{A_i})\cong \mathbb{Z}$ for every $i \in \{1,\ldots,r\}$ by Poincar\'{e} duality. So it is enough to show that 
$$(1,0,\ldots,0) \in  \mathrm{H}_c^{\vcd(\Gamma)}(X^{A_1})\oplus\dots\oplus \mathrm{H}_c^{\vcd(\Gamma)}(X^{A_r})$$
lies in the image of the map (\ref{eq: map}). To this end, let $F$ be a rational flat in $X$ that intersects $X^{A_1}$ in exactly one point and is disjoint from $X^{A_i}$ for $i \in \{2,\ldots,r\}$. Letting $q$ equal the $\mathbb{Q}$-rank of $\Gamma$  (i.e.~$\rank_\BQ(\BG)$), one checks using Proposition \ref{prop: flat cohom class} and  Poincar\'{e}-Lefschetz duality that
\begin{align*}
\mathbb{Z} &\cong\mathrm{H}_{q}(F,\partial F) \rightarrow \mathrm{H}_{q}(X,\partial X) \cong \mathrm{H}_c^{\vcd(\Gamma)}(X) \rightarrow \\
& \rightarrow \mathrm{H}_c^{\vcd(\Gamma)}(X^{A_1})\oplus\dots\oplus \mathrm{H}_c^{\vcd(\Gamma)}(X^{A_r})\cong\BZ^r 
\end{align*}
sends $1$ to $(1,0,\ldots,0)$, finishing the proof.
\end{proof}

Corollary \ref{crit1} and Corollary \ref{cor: cohom flat} will suffice to prove that $\cdm(\Gamma)=\vcd(\Gamma)$ for every lattice $\Gamma$ in one of the Lie groups $G$ considered in the main theorem other than $G=\SL_3\BR$. The proof in this particular case will rely on the following proposition.

\begin{prop}\label{gen} 
With the same notation as above, assume that for every finite subgroup $H$ of $\Gamma$ that properly contains the center of $\Gamma$, one has
$$\mathrm{H}_c^n(X^H,X^H_{\sing})=0$$
for every $n\geq \mathrm{vcd}(\Gamma).$ Then $\vcd(\Gamma)=\cdm(\Gamma)$.
\end{prop}

\begin{proof} 
Recall that $\CF$ is the collection of finite subgroups of $\Gamma$ that contain the center of $\Gamma$. Define the length $l(H)$ of a finite subgroup $H$ of $\Gamma$ to be the largest integer $m>0$ such that there is a strictly ascending chain of subgroups
$$Z(\Gamma)=H_0\subsetneq H_1\subsetneq\ldots\subsetneq H_m=H.$$
where $Z(\Gamma)$ is the center of $\Gamma$. Note that since $\Gamma$ is virtually torsion-free, there is a uniform bound on the length of finite subgroups of $\Gamma$. Hence, the number
\[  l=\max\{ l(H) \ | \ H \in \CF\}  \]
is finite. We use the notion of length to filter $X_{\sing}$ as follows. For each $i \in \{0,\ldots,l\}$, define

$$X_i=\bigcup_{\substack{H \in \CF s.t.\\ l(H)=i} }X^H.$$

Note that since $X^L\subseteq X^H$ whenever $H$ is a subgroup of  $L$, we have $X_1=X_{\sing}$ and 
$$X_l\subseteq X_{l-1}\subseteq\ldots\subseteq X_{i+1}\subseteq X_i\subseteq\ldots\subseteq X_1\subseteq X_0=X.$$
If $H_1$ and $H_2$ are two distinct finite subgroups of $\Gamma$ of maximal length $l$, then $X^{H_1}\cap X^{H_2}=X^{\langle H_1,H_2\rangle}=\emptyset$, since $\langle H_1,H_2\rangle$ must be an infinite group. Therefore $X_l$ is a disjoint union of spaces of the form $X^H$ with $l(H)=l$. Since $X^H_{\mathrm{sing}}$ is empty when $l(H)=l$ we obtain
\begin{equation}\label{maxlength}H^n_c(X_l)=\bigoplus_{l(H)=l}H^n_c(X^H)=\bigoplus_{l(H)=l}H^n_c(X^H,X^H_{\mathrm{sing}})=0.\end{equation}
for every $n\geq \vcd(\Gamma)$. Since for any $i\in \{1,\ldots,l\}$ and each finite subgroup $H \in \CF$ of length $l(H)=i-1$
$$X^H\cap X_i=X^H_{\sing},$$ 
Lemma \ref{covering} implies that
\begin{equation}\label{split}H^*_c(X_{i-1},X_i)=\bigoplus_{l(H)=i-1}H^*_c(X^H,X^H_{\sing}).\end{equation}
We claim that for any $i\in \{1,\ldots,l\}$ and $n\geq\vcd\Gamma$, we have $H^{n+1}_c(X,X_{i})=0$. To prove this claim, we argue by induction starting at $X_l$ and going down. The long exact sequence of the pair $(X_l,X)$ together with $(\ref{maxlength})$ implies that $H^{n+1}_c(X,X_l)=0$.  Proceeding inductively, assume that $H^{n+1}(X,X_i)=0$. Now the long exact sequence of the triple $(X,X_{i-1},X_i)$ together with  (\ref{split})  implies that $H^{n+1}_c(X,X_{i-1})=0$, proving the claim. Since $X_1=X_{\mathrm{sing}}$ we have $H^{n+1}_c(X,X_{\mathrm{sing}})=0$ for all $n\geq \vcd(\Gamma)$. The proposition now follows from Theorem \ref{th: bredon compact support}.
\end{proof}

%
%
%
%

\section{Classical Lie groups}\label{sec: classical groups}
In this section we recall the definitions of the classical Lie groups and discuss a few well-known facts about their maximal compact subgroups and ranks. There are many references for this material, see for example \cite{Knapp}.

We begin by introducing some notation which we will keep using throughout the whole paper. We denote the transpose of a matrix $A$ by $A^t$. If $A$ is a complex (resp. quaternionic) matrix, we denote by $A^*$ its conjugate transpose. Accordingly, we sometimes denote the transpose $A^t$ of a real matrix $A$ by $A^*$. Now consider the block matrices
$$\Id_n=
\left(
\begin{array}{ccc}
1  &   &   \\
  &  \ddots &   \\
  &   &  1
\end{array}
\right),\ \ \ 
J_n=\left(
\begin{array}{cc}
 & \Id_n \\
-\Id_n & 
\end{array}
\right),\ \ 
Q_{p,q}=
\left(
\begin{array}{cc}
-\Id_p &  \\
 & \Id_q
\end{array}
\right)$$
where the empty blocks are zero. If the dimensions are understood, then we will drop them from our notations. 
The following groups are known as the classical (non-compact) simple Lie groups
\begin{align*}
&\SL(n,\BC)=\{A\in\GL(n,\BC)\vert \det A=1\} & & n\ge 2 \\
& \SO(n,\BC)=\{A\in\SL(n,\BC) \vert A^tA=\Id\} & & n\geq 3,\ n\neq 4 \\
&\Sp(2n,\BC)=\{A\in\SL(2n,\BC) \vert A^tJ_nA=J_n\} & & n\ge 1 \\
&\SL(n,\BR)=\{A\in\GL(n,\BR)\vert \det A=1\} & & n\ge 2 \\
&\SL(n,\BH)=\{A\in\GL(n,\BH)\vert \det A=1\} & & n\ge 2 \\
&\SO(p,q)=\{A\in\SL(p+q,\BR)\vert A^*Q_{p,q}A=Q_{p,q}\}& & 1\le p\le q,\ p+q\ge 3 \\
&\SU(p,q)=\{A\in\SL(p+q,\BC)\vert A^*Q_{p,q}A=Q_{p,q}\}& & 1\le p\le q,\ p+q\ge 3 \\
&\Sp(p,q)=\{A\in\GL(p+q,\BH)\vert A^*Q_{p,q}A=Q_{p,q}\}& & 1\le p\le q,\ p+q\ge 3 \\
&\Sp(2n,\BR)=\{A\in\SL(2n,\BR)\vert A^tJ_nA=J_n\} & & n\ge 1 \\
&\SO^*(2n)=\{A\in\SU(n,n)\vert A^tQ_{n,n}J_nA=Q_{n,n}J_n\} & & n\ge 2
\end{align*}

If $k=\BR,\BC$ and $G\subset\GL(n,k)$ then we let \[SG=\{A\in G\vert \det(A)=1\}\] be the set of elements with unit determinant. However, it is important to know what `$\det$' means because the real determinant of a matrix $A\in M_n(\BC)\subset M_{2n}(\BR)$ is the square of the norm of the complex determinant. For instance, every unitary matrix has determinant 1 when considered as a matrix with real coefficients: $\UU_n\subset\SL(2n,\BR)$. The reader might also wonder what is meant by $\SL(n,\BH)$ because a quaternionic endomorphism has no canonical determinant with values in $\BH$. Whenever we write $\det A$ for $A\in M_n(\BH)$ we consider the complex determinant of the image of $A$ under the embedding $M_n(\BH)\subset M_{2n}(\BC)$ obtained through the identification of $\BH$ with the following subalgebra of $M_2(\BC)$
$$\BH=\left\{
\left(
\begin{array}{cc}
z_1  & -\bar z_2  \\
z_2  & \bar z_1  
\end{array}
\right)\middle\vert z_1,z_2\in\BC\right\}.$$
When doing so, we get an identification between $\SL(n,\BH)$ and the group 
$$\SU^*(2n)=\left\{
\left(
\begin{array}{cc}
Z_1  & -\bar Z_2  \\
Z_2  & \bar Z_1  
\end{array}
\right)\in\SL(2n,\BC)\right\}.$$
Besides the classical (non-compact) Lie groups listed above, the compact ones are also going to play a key role here. Each of the classical Lie groups $G$ has a unique maximal compact group $K$ up to conjugacy. They are all basically constructed out of the following individual groups
\begin{align*}
& \SO_n=\{A\in\SL(n,\BR)\vert A^*A=\Id\} & { } & \OO_n=\{A\in\GL(n,\BR)\vert A^*A=\Id\}\\
& \SU_n=\{A\in\SL(n,\BC)\vert A^*A=\Id\} & { } & \UU_n=\{A\in\GL(n,\BC)\vert A^*A=\Id\}\\
& \Sp_n=\{A\in\GL(n,\BH)\vert A^*A=\Id\}. & & 
\end{align*}
Note that subindices will be used exclusively for compact groups. For later use we record here the dimensions of these groups as real Lie groups
$$\dim\SO_n=\frac{n(n-1)}2,\ \dim\SU_n=n^2-1,\ \dim\Sp_n=2n^2+n.$$
The dimensions of the other compact groups which appear can be computed easily from these numbers.\\

Later on we will be interested in the dimensions of classical Lie groups, their maximal compact subgroups, the dimensions of the associated symmetric space $G/K$ and the real rank of $G$. We organize this data in the following two tables.
\begin{table}[htdp]
\begin{center}
\caption{Classical Lie groups $G$, maximal compact subgroups $K$ and dimensions as real Lie groups.}\label{Tabla1}
\begin{tabular}{|c|c|c|c|c|c|c|}
$G$ & $K$ & $\dim(G)$ & $\dim(K)$  \\
\hline
$\SL(n,\BC)$ & $\SU_n$ & $2(n^2-1)$ & $n^2-1$  \\
\hline
$\SO(n,\BC)$ & $\SO_n$ & $n(n-1)$ & $\frac{n(n-1)}2$  \\
\hline
$\Sp(2n,\BC)$ & $\Sp_n$ & $2n(2n+1)$ & $n(2n+1)$  \\
\hline
$\SL(n,\BR)$ & $\SO_n$ & $n^2-1$ & $\frac{n(n-1)}2$  \\
\hline
$\SL(n,\BH)$ & $\Sp_n$ & $4n^2-1$ & $n(2n+1)$  \\
\hline
$\SO(p,q)$ ($p\le q$) & $S(\OO_p\times\OO_q)$ & $\frac{(p+q)(p+q-1)}2$ & $\frac{p^2+q^2-p-q}2$  \\
\hline
$\SU(p,q)$ $(p\le q$) & $S(\UU_p\times\UU_q)$ & $(p+q)^2-1$ & $p^2+q^2-1$  \\
\hline
$\Sp(p,q)$ $(p\le q$) & $\Sp_q\times\Sp_q$ & $(p+q)(2p+2q+1)$ & $2p^2+2q^2+p+q$  \\
\hline
$\Sp(2n,\BR)$ & $\UU_n$ & $2n^2+n$ & $n^2$  \\
\hline
$\SO^*(2n)$ & $\UU_n$ & $2n^2-n$ & $n^2$ 
\end{tabular}
\end{center}

\end{table}
\begin{table}[htdp]
\caption{Classical groups $G$, the dimensions of symmetric space $G/K$ and the real rank of $G$. Here and in the sequel $\left[x\right]$ stands for the integer part of $x$.}\label{Tabla2}
\begin{center}
\begin{tabular}{|c|c|c|c|c|c|c|}
$G$ & $\dim(G/K)$ & $\rank_\BR(G)$ \\
\hline
$\SL(n,\BC)$ & $n^2-1$ & $n-1$ \\
\hline
$\SO(n,\BC)$ & $\frac{n(n-1)}2$ & $\left[\frac n2\right]$\\
\hline
$\Sp(2n,\BC)$ & $n(2n+1)$ & $n$ \\
\hline
$\SL(n,\BR)$ & $\frac{n(n-1)}2+n-1$ & $n-1$ \\
\hline
$\SL(n,\BH)$ & $2n^2-n-1$ & $n-1$ \\
\hline
$\SO(p,q)$ ($p\le q$) & $pq$ & $p$ \\
\hline
$\SU(p,q)$ $(p\le q$) & $2pq$ & $p$ \\
\hline
$\Sp(p,q)$ $(p\le q$) & $4pq$ & $p$ \\
\hline
$\Sp(2n,\BR)$ & $n^2+n$ & $n$ \\
\hline
$\SO^*(2n)$ & $n^2-n$ & $\left[\frac n2\right] $
\end{tabular}
\end{center}
\end{table}

\newpage
Armed with the data from Table \ref{Tabla2} we get from Lemma \ref{borel-serre} some lower bounds for the virtual cohomological dimension of lattices in classical groups.

\begin{prop}\label{vcd lower bound}
Let $\Gamma\subset G$ be a lattice.

\begin{itemize}
\item If $G=\SL(n,\BC)$, then $\vcd(\Gamma)\ge n^2-n$.
\item If $G=\SO(n,\BC)$, then $\vcd(\Gamma)\ge \frac{n(n-1)}2-\left[\frac n2\right]$.
\item If $G=\Sp(2n,\BC)$, then $\vcd(\Gamma)\ge 2n^2$.
\item If $G=\SL(n,\BR)$, then $\vcd(\Gamma)\ge \frac{n(n-1)}2$.
\item If $G=\SL(n,\BH)$, then $\vcd(\Gamma)\ge 2n^2-2n$.
\item If $G=\SO(p,q)$ with $p\le q$, then $\vcd(\Gamma)\ge pq-p$.
\item If $G=\SU(p,q)$ with $p\le q$, then $\vcd(\Gamma)\ge 2pq-p$.
\item If $G=\Sp(p,q)$ with $p\le q$, then $\vcd(\Gamma)\ge 4pq-p$.
\item If $G=\Sp(2n,\BR)$, then $\vcd(\Gamma)\ge n^2$.
\item If $G=\SO^*(2n)$, then $\vcd(\Gamma)\ge n^2-n-\left[\frac n2\right]$.
\end{itemize}
In all cases equality happens if and only if $\rank_\BQ(\Gamma)=\rank_\BR(G)$.\qed
\label{prop-vcd}
\end{prop}

Also, recall that by Corollary \ref{crit0} we have $\vcd(\Gamma)=\gdim(\Gamma)$ for every lattice $\Gamma$ in a group with real rank $1$. Hence, we obtain the following.

\begin{kor}\label{case0}
If $\Gamma$ is a lattice in 
\begin{itemize}
\item $\SO(1,q)$, $\SU(1,q)$ or $\Sp(1,q)$ for some $q$, or in
\item $\SL(2,\BC)$, $\SO(3,\BC)$, $\Sp(2,\BC)$, $\SL(2,\BR)$, $\SL(2,\BH)$, $\Sp(2,\BR)$, $\SO^*(4)$, $\SO^*(6)$, 
\end{itemize}
then $\gdim(\Gamma)=\vcd(\Gamma)$.\qed
\end{kor}

%
%
%
%

\section{Centralizers in compact groups}\label{sec:compact}

In this section we recall how the centralizers of elements in the compact groups $\SO_n,\SU_n$ and $\Sp_n$ look like.

\subsection{Unitary groups}
The diagonal subgroup
$$T=\left\{\left(
\begin{array}{ccc}
\lambda_1 & & \\
& \ddots & \\
& & \lambda_n
\end{array}
\right)\middle| \lambda_i\in\BS^1\right\}$$
is a maximal torus in $\UU_n$. Since $\UU_n$ is connected, every element $A\in\UU_n$ is contained in a maximal torus, and since any two maximal tori are conjugate, it follows that any $A\in\UU_n$ can be conjugated to some $A'\in T$ and that $C_{\UU_n}(A)$ and $C_{\UU_n}(A')$ are also conjugate.

Given $A\in T,$ let $\lambda$ be an eigenvalue and  $E_\lambda^A$ the associated eigenspace. Note that $E_\lambda^A$ and $E_\mu^A$ are orthogonal for any two distinct $\lambda\neq\mu$. Observe also that each  of the eigenspaces $E^A_\lambda$ is invariant under any element in $\UU_n$ which commutes with $A$. On the other hand, every element in $\UU_n$ which preserves the eigenspaces $E_\lambda^A$ commutes with $A$. Altogether we get that the centralizer $C_{\UU_n}(A)$ of $A$ consists precisely of those elements in $\UU_n$ which preserve the eigenspaces $E_\lambda^A$. If the said eigenspaces are $E_{\lambda_1}^A,\dots,E_{\lambda_r}^A$ and  $\dim E_{\lambda_i}^A=d_i$, then again up to conjugation in $\UU_n$ we can assume that $E_{\lambda_i}^A$ is spanned by the vectors $e_j$ in the standard basis with $j\in\{1+\sum_{k<i}d_k,\ldots,\sum_{k\le i}d_k\}$. This just means that any element centralizing $A$ has a block form
$$\left(
\begin{array}{ccc}
O_1 & & \\
& \ddots & \\
& & O_r
\end{array}
\right)$$
with $O_i$ a $d_i$-square matrix. Altogether we have the following well-known fact that we state as a lemma for future reference.

\begin{lem}\label{centralizer-U}
The centralizer $C_{\UU_n}(A)$ of any element $A\in\UU_n$ is conjugate within $\UU_n$ to the group $\UU_{d_1}\times\dots\times\UU_{d_r}$ where the numbers $d_1,\dots,d_r$ are the dimensions of the different eigenspaces of $A$.
\end{lem}

Since any two elements in $\SU_n$ which are conjugate within $\UU_n$ are also conjugate within $\SU_n$, we also have the corresponding result for special unitary groups.

\begin{lem}\label{centralizer-SU}
The centralizer $C_{\SU_n}(A)$ of any element $A\in\SU_n$ is conjugate within $\SU_n$ to the group $S(\UU_{d_1}\times\dots\times\UU_{d_r})$ where the numbers $d_1,\dots,d_r$ are the dimensions of the different eigenspaces of $A$.
\end{lem}

\subsection{Special orthogonal group $\SO_n$}\label{centSO}
Since the group $\SO_n$ is connected, we can again conjugate any element in $\SO_n$ into a particular maximal torus $T$ of our choosing. For $n=2k$ the group
$$T=\left\{\left(
\begin{array}{ccc}
O_1 & & \\
& \ddots & \\
& & O_k
\end{array}
\right)\middle| O_i\in\SO_2\right\}$$
is a maximal torus of $\SO_n$. For $n=2k+1$ a maximal torus is the group
$$T=\left\{\left(
\begin{array}{cccc}
O_1 & & & \\
& \ddots & & \\
& & O_k & \\
& & & 1
\end{array}
\right)\middle| O_i\in\SO_2\right\}.$$
In all cases we get that any element in $\SO_n$ is conjugate to an element of the form
$$A=\left(
\begin{array}{ccccc}
A_1 & & & & \\
& \ddots & & & \\
& & A_r & & \\
& & & \Id_s & \\
& & & & -\Id_t
\end{array}\right)$$
where each $A_i$ is a $2d_i$-square matrix
$$A_i=\left(
\begin{array}{ccc}
O_i & & \\
& \ddots & \\
& & O_i
\end{array}
\right)$$
with $O_i\in\SO_2$ but $O_i\neq\pm\Id$ and such that $O_i\neq O_j$ for $i\neq j$. Any element in $\OO_n$ which commutes with $A$  preserves these blocks, meaning that
$$C_{\OO_n}(A)=C_{\OO_{2d_1}}(A_1)\times\dots\times C_{\OO_{2d_r}}(A_r)\times\OO_s\times\OO_t$$
Observing that the element $A_i\in\SO_{2d_i}$ preserves an (essentially unique) complex structure on $\BR^{2d_i}$, one deduces that every element which commutes with it has to also preserve it. This means that in fact, if we take the usual identification between $\BC^{d_i}$ and $\BR^{2d_i}$ we have that $C_{\OO_{2d_i}}(A_i)=\UU_{d_i}$. Altogether we have the following.

\begin{lem}\label{centralizer-SO}
The centralizer $C_{\SO_n}(A)$ of any element $A\in\SO_n$ is conjugate within $\SO_n$ to the group $S(\UU_{d_1}\times\dots\times\UU_{d_r}\times\OO_s\times\OO_t)$ where $s$ and $t$ are respectively the dimensions of the $\pm 1$-eigenspaces and where $n=2d_1+\dots+2d_r+s+t$.
\end{lem}

\subsection{Compact symplectic groups $\Sp_n$}
The matrix
$$\hat J_n=\left(
\begin{array}{ccc}
J_1 & & \\
 & \ddots & \\
 & & J_1
\end{array}
\right)=\left(
\begin{array}{ccccc}
0 & 1 & & & \\
-1 & 0 & & & \\
& & \ddots & & \\
& & & 0 & 1 \\
& & & -1 & 0
\end{array}
\right)$$
is conjugate within $\UU_{2n}$ to $J_n$ by a matrix with real values. In particular, we obtain that the compact symplectic group $\Sp_n=\Sp(2n,\BC)\cap\UU_{2n}$ is conjugate within $\UU_{2n}$ to the group 
$$\widehat\Sp_n=\{A\in\UU_{2n}\vert A^t\hat J_n A=\hat J_n\}$$
This later incarnation of the compact symplectic group has some advantages. For instance, it is easier to see it as a subgroup of the group of automorphisms of $n$-dimensional quaternionic space $\BH^n$ (where we multiply scalars in $\BH^n$ from the right). Indeed, recalling the standard identification of the quaternions $\BH$ with the following subalgebra of $M_2(\BC)$
$$\BH=\left\{
\left(
\begin{array}{cc}
z  & -\bar w \\
w  & \bar z
\end{array}
\right)\middle| z,w\in\BC\right\},$$
we get an identification
$$\GL(n,\BH)=\left\{
\left(
\begin{array}{ccc}
q_{1,1} & \dots & q_{1,n} \\
\vdots & \ddots & \vdots \\
q_{n,1} & \dots & q_{n,n}
\end{array}
\right)\in\GL(2n,\BC)\middle| q_{i,j}\in\BH\right\}.$$
Noting that $\BH$ consists exactly of those elements $A\in M_2\BC$ with $J_1A=\bar AJ_1$, and taking into account that for every $A\in\widehat\Sp_{2n}$ we have at the same time that $A^*=A^{-1}$ and that $A^t\hat J_nA=\hat J_n$, one gets that $\widehat\Sp_n\subset\GL_n\BH$. Hoping that no confusion will occur, we will from now on denote both versions, $\widehat\Sp_n$ and $\Sp_n$, of the compact symplectic group by the same symbol $\Sp_n$.

After these preliminary remarks we proceed as above. As in the previous cases, the connectedness of $\Sp_n$ implies that every element in $\Sp_n$ can be conjugated into any maximal torus. In this case a maximal torus is given by the group
$$T=\left\{
\left(
\begin{array}{ccccc}
\lambda_1 & 0 & & & \\
0 & \bar\lambda_1 & & & \\
& & \ddots & & \\
& & & \lambda_n & 0 \\
& & & 0 & \bar\lambda_n
\end{array}
\right)
\middle| \lambda_i\in\BC\ \hbox{with}\ \vert\lambda_i\vert=1\right\}.$$
This implies in particular that if $\lambda$ is an eigenvalue of $A$, then $\bar\lambda$ is also an eigenvalue. For any eigenvalue $\lambda$ let 
\begin{equation}\label{eq-quater-eigen}
V_\lambda=E_{\lambda} + E_{\bar\lambda}
\end{equation}
be the sum of the eigenspaces corresponding to $\lambda,\bar\lambda$. This is a direct sum unless $\lambda=\pm 1$. We recall a few properties of the spaces $V_\lambda$.
\begin{itemize}
\item $V_\lambda$ is a quaternionic subspace of $\BH^n$.
\item For $\lambda\neq\mu,\bar\mu$, the spaces $V_\lambda$ and $V_\mu$ are orthogonal with respect to the standard unitary scalar product of $\BC^{2n}$.
\item The symplectic form $(v,w)\mapsto v^t J_nw$ (resp. $(v,w)\mapsto v^t\hat J_nw$) on $\BC^{2n}$ restricts to a non-degenerate symplectic form on each $V_\lambda$. 
\end{itemize}

In particular, every element in $\Sp_n$ which centralizes $A$ restricts to an element in $\Sp(V_\lambda)$ for each eigenvalue $\lambda$. It follows that if denote the eigenvalues of $A$ by $\lambda_1,\dots,\lambda_r$,  let $2d_i=\dim V_{\lambda_i}$ and set
\begin{equation}\label{blabla}
\Lambda_i=
\left(
\begin{array}{ccccc}
\lambda_i & 0 & & & \\
0 & \bar\lambda_i & & & \\
& & \ddots & & \\
& & & \lambda_i & 0 \\
& & & 0 & \bar\lambda_i
\end{array}
\right) \in \Sp_{d_i}
\end{equation}
we get that the centralizer of $A$ in $\Sp_n$ can be conjugated within $\Sp_n$ to
$$C_{\Sp_n}(A)\stackrel{\hbox{\tiny conj.}}\simeq C_{\Sp_{d_1}}(\Lambda_1)\times\dots\times C_{\Sp_{d_r}}(\Lambda_r)$$
Continuing with the same notation, suppose that $\lambda_i$ is not real. In this case the element $\Lambda_i$ is not central in $\GL(d_i,\BH)$. In fact, its centralizer is the subgroup consisting of quaternionic matrices with coefficients in the subalgebra
$$\BC=\left\{
\left(
\begin{array}{cc}
z  & 0 \\
0  & \bar z
\end{array}
\right)\middle| z\in\BC\right\}\subset\BH.$$
It follows that the centralizer of $\Lambda_i$ in $\GL(d_i,\BH)$ is the subgroup $\GL(d_i,\BC)$, which leads to
$$C_{\Sp_{d_i}}(\Lambda_i)\simeq\UU_{d_i}.$$
All this implies the following.
\begin{lem}\label{centralizer-Sp}
The centralizer $C_{\Sp_n}(A)$ of any element $A\in\Sp_n$ is isomorphic to the group $\UU_{d_1}\times\dots\times\UU_{d_r}\times\Sp_s\times\Sp_t$ where $s$ and $t$ are respectively the dimensions of the $\pm 1$-eigenspaces and where $n=d_1+\dots+d_r+s+t$.
\end{lem}
For later use, we remark that with the same notation as in Lemma \ref{centralizer-Sp} we also have
$$C_{\GL(n,\BH)}(A)\simeq\GL(d_1,\BC)\times\dots\times\GL(d_r,\BC)\times\GL(s,\BH)\times\GL(t,\BH).$$

%
%
%
%

\section{Easy cases}\label{sec:easy}
In this section we will prove that $\cdm(\Gamma)=\vcd(\Gamma)$ for any lattice $\Gamma$ in a classical group belonging to one of the following $8$ families: $\SL(n,\BC)$, $\SO(n,\BC)$, $\Sp(2n,\BC)$, $\SL(n,\BH)$, $\SU(p,q)$, $\Sp(p,q)$, $\Sp(2n,\BR)$, $\SO^*(2n)$. 

The cases treated here are relatively simple because in all of them we have that the dimension of the fixed point set $S^A$ of any non-central element $A\in\Gamma$ is smaller than the virtual cohomological dimension of $\Gamma$. This implies that the desired conclusion follows from Corollary \ref{crit1}. As before, $S$ is the symmetric space associated to the simple Lie group containing $\Gamma$.

\subsection{Complex special linear group $\SL(n,\BC)$}
Consider the symmetric space $S=\SL(n,\BC)/\SU_n$ and let $A\in\SL(n,\BC)$ be a non-central finite order element. The matrix $A$ can be conjugated into the maximal compact subgroup $\SU_n$. Moreover, we get from Corollary \ref{complexification} that
$$\dim S^A=\dim C_{\SU_n}(A)$$
where the dimension is that of the centralizer as a real Lie group. Now, by Lemma \ref{centralizer-SU} we have that up to conjugation in $\SU_n$
$$C_{\SU_n}(A)=S(\UU_{d_1}\times\dots\times\UU_{d_r})$$
where $\sum_{i=1}^r d_i=n$ and where $r\ge 2$ as $A$ is not central. It follows that
$$\dim C_{\SU_n}(A)=-1+\sum_i\dim\UU_{d_i}=-1+\sum_id_i^2$$
Now, for $A$ non-central, this quantity is maximized if $r=2$ and $d_1=1$, meaning that
$$\dim S^A=\dim C_{\SU_n}(A)\le(n-1)^2$$
Now, if $\Gamma\subset\SL(n,\BC)$ is a lattice we have that 
$$\vcd(\Gamma)\ge n^2-n=(n-1)^2+(n-1)$$
by Proposition \ref{vcd lower bound}. Corollary \ref{crit1} applies and we obtain the following.

\begin{lem}\label{done-slnc}
If $\Gamma$ is a lattice in $\SL(n,\BC)$ for $n\geq 2 $, then $\cdm\Gamma=\vcd(\Gamma)$.\qed
\end{lem}

\subsection{Complex special orthogonal groups $\SO(n,\BC)$}\label{complexorthogonal}
As always we denote by $S=\SO(n,\BC)/\SO_n$ the relevant symmetric space. Any finite order element $A\in\SO(n,\BC)$ can be conjugated into the maximal compact subgroup $\SO_n$, and again we obtain from Corollary \ref{complexification} that
$$\dim S^A=\dim C_{\SO_n}(A).$$
From Lemma \ref{centralizer-SO} one gets that $C_{\SO_n}(A)$ is conjugate within $\SO_n$ to the group $S(\UU_{d_1}\times\dots\times\UU_{d_r}\times\OO_s\times\OO_t)$ where $s$ and $t$ are respectively the dimensions of the $\pm 1$-eigenspaces and where $n=2d_1+\dots+2d_r+s+t$. It follows that 
\begin{equation}\label{eq-copen1}
\dim C_{\SO_n}(A)=\sum_{i=1}^rd_i^2+\frac{s(s-1)}2+\frac{t(t-1)}2
\end{equation}
It is easy to see that as long as $n\ge 5$, this quantity is maximal only if $r=0$ and either $s$ or $t$ is equal to $1$, meaning that
$$\dim S^A\le \frac{(n-1)(n-2)}2=\frac{n(n-3)}2+1\ \ \hbox{if}\ n\ge 5.$$
Now, if $\Gamma\subset\SO(n,\BC)$ is a lattice then we have that 
$$\vcd(\Gamma)\ge \frac{n(n-2)}2$$
by Proposition \ref{vcd lower bound}. Taken together, these two inequalities imply that, as long as $n\ge 5$, the assumption of Corollary \ref{crit1} is satisfied.

\begin{lem}\label{done-sonc}
If $n\ge 5$ and $\Gamma\subset\SO(n,\BC)$ is a lattice, then $\cdm\Gamma=\vcd(\Gamma)$.\qed
\end{lem}

\subsection{Complex symplectic group $\Sp(2n,\BC)$}
Once more we have that every finite order element $A\in\Sp(2n,\BC)$ can be conjugated into the maximal compact group $\Sp_n$, and again we get from Corollary \ref{complexification} that
$$\dim S^A=\dim C_{\Sp_n}(A)$$
where $S=\Sp(2n,\BC)/\Sp_n$ is the symmetric space of $\Sp(2n,\BC)$. From Lemma \ref{centralizer-Sp} we get that $C_{\Sp_n}(A)$ is isomorphic to a group of the form 
$$\UU_{d_1}\times\dots\times\UU_{d_r}\times\Sp_s\times\Sp_t$$
where $s$ and $t$ are respectively the dimensions of the $\pm 1$-eigenspaces and where $n=d_1+\dots+d_r+s+t$. Moreover,   $A$ is non-central if and only if $s,t\neq n$.
Computing dimensions we get that
$$\dim C_{\Sp_n}(A)=\sum_{i=1}^r d_i^2+2s^2+s+2t^2+t$$
Now, assuming that $A$ is non-central, it is not hard to check that this expression is maximal if $r=0$, $s=n-1$ and $t=1$, meaning that 
$$\dim S^A\le 2n^2-3n+4$$
From Proposition \ref{vcd lower bound} we get the lower bound $\vcd(\Gamma)\ge 2n^2$ for any lattice $\Gamma\in\Sp(2n,\BC)$, meaning that again Corollary \ref{crit1} applies and yields the following.

\begin{lem}\label{done-sp2nc}
If $n\ge 2$ and $\Gamma\subset\Sp(2n,\BC)$ is a lattice, then $\vcd(\Gamma)=\cdm\Gamma$.\qed
\end{lem}

\subsection{Quaternionic special linear group $\SL(n,\BH)$}\label{quaternionic} 
As always we can conjugate every finite order element $A\in\SL(n,\BH)$ into the maximal compact subgroup $\Sp_n$. As mentioned after Lemma \ref{centralizer-Sp} we have 
$$C_{\GL(n,\BH)}(A)\simeq\GL(d_1,\BC)\times\dots\times\GL(d_r,\BC)\times\GL(s,\BH)\times\GL(t,\BH)$$
for each $A\in\Sp_n$. Note also that $A$ is non-central if and only if $s,t\neq n$. From this, Proposition \ref{cartan} and Lemma \ref{centralizer-Sp} we obtain that
$$\dim S^A\le d_1^2+\dots+d_r^2+2s^2-s+2t^2-t-1$$
where $S=\SL(n,\BH)/\Sp_n$ is the relevant symmetric space. This quantity is maximal if $r=0$ and $(s,t)=(1,n-1)$ or $(s,t)=(n-1,1)$, meaning that
$$\dim S^A\le 2n^2-5n+3.$$
On the other hand Proposition \ref{vcd lower bound} yields the lower bound $\vcd(\Gamma)\ge 2n^2-2n$ for any lattice $\Gamma\in\SL(n,\BH)$, which means that again Corollary \ref{crit1} implies the following

\begin{lem}\label{done-slnh}
For every lattice $\Gamma$ in $\SL(n,\BH)$ we have $\vcd(\Gamma)=\cdm\Gamma$.
\end{lem}

\subsection{Indefinite unitary group $\SU(p,q)$}
Any finite order element $A\in\SU(p,q)$ can be conjugated into the maximal compact subgroup $S(\UU_p\times\UU_q)$, so we can assume that $A$ in contained in $S(\UU_p\times\UU_q)$ from the start. Note that if $\lambda$ is an eigenvalue of $A$, then for the eigenspace $E_\lambda$ we have
$$E_\lambda=(E_\lambda\cap(\BC^p\times\{0\}))\oplus(E_\lambda\cap(\{0\}\times\BC^q)).$$
This means that the hermitian form $(v,w)\mapsto\bar v^tQ_{p,q}w$ induces a hermitian form on $E_\lambda$ of signature $(\dim_\BC E_\lambda\cap(\BC^p\times\{0\}),\dim_\BC E_\lambda\cap(\{0\}\times\BC^q))$. This latter form is preserved by any element in $\SU(p,q)$ preserving $E_\lambda$, and hence in particular by the centralizer of $A$. Taking all  this together we deduce that, up to conjugation,
$$C_{\SU(p,q)}(A)=S(\UU(p_1,q_1)\times\dots\times\UU(p_r,q_r))).$$
where $p=\sum p_i$ and $q=\sum q_i$. Together with Lemma \ref{centralizer-SU} this implies that
$$\dim S^A\le \sum_{i=1}^r 2p_iq_i$$
where $S=\SU(p,q)/S(\UU_p\times\UU_q)$. If $A$ is non-central we have that $r\ge 2$ and hence 
$$\dim S^A\le 2p(q-1).$$
Since we have by Proposition \ref{vcd lower bound} that $\vcd(\Gamma)\ge 2pq-p$ for any lattice $\Gamma \subset \SU(p,q)$, we get the following from Corollary \ref{crit1}.

\begin{lem}\label{done-supq}
Let $p \ge 2$. If $\Gamma$ is a lattice in $\SU(p,q)$ then $\vcd(\Gamma)=\cdm\Gamma$.\qed
\end{lem}

\subsection{Indefinite quaternionic unitary group $\Sp(p,q)$}
Arguing as usual, note that any finite order $A\in\Sp(p,q)\subset\GL(2p+2q,\BC)$ can be conjugated into the maximal compact subgroup $\Sp_p\times\Sp_q\subset\UU_{2p}\times\UU_{2q}$. So we can assume that $A$ is contained in $\Sp_p\times\Sp_q\subset\UU_{2p}\times\UU_{2q}$ to begin with.  As in the case of $\SU(p,q)$ note that if $\lambda\in\BC$ is an eigenvalue of $A$, then for the eigenspace $V_\lambda$ as in \eqref{eq-quater-eigen} we have
$$V_\lambda=(V_\lambda\cap(\BH^{p}\times\{0\}))\oplus(V_\lambda\cap(\{0\}\times\BH^{q})).$$
This again implies that the hermitian form $(v,w)\mapsto\bar v^tQ_{p,q}w$ induces a hermitian form on $V_\lambda$ of signature $$(\dim_\BH V_\lambda\cap(\BH^p\times\{0\}),\dim_\BH V_\lambda\cap(\{0\}\times\BH^q)).$$ This latter form is preserved by any element in $\Sp(p,q)$ preserving $V_\lambda$, and hence in particular by the centralizer of $A$. Taking all of this together and using the same notation as in \eqref{blabla}, we have that $C_{\Sp(p,q)}(A)$ is conjugate to
$$C_{\Sp(p,q)}(A)=C_{\Sp(p_1,q_1)}(\Lambda_1)\times\dots\times C_{\Sp(p_r,q_r)}(\Lambda_r)\times\Sp(p_s,q_s)\times\Sp(p_t,q_t)$$
where 
\begin{itemize}
\item $\lambda_1,\bar\lambda_1,\dots,\lambda_r,\bar\lambda_r$ are the complex non-real eigenvalues of $A$,
\item $p=p_s+p_t+\sum p_i$ and $q=q_s+q_t+\sum q_i$. 
\end{itemize}
Compare with Lemma \ref{centralizer-Sp}. Taking into account that
\begin{align*}
C_{\Sp(p_i,q_i)}(\Lambda_i)
   &=\Sp(p_i,q_i)\cap C_{\GL(p_i+q_i,\BH)}(\Lambda_i)\\
   &=\Sp(p_i,q_i)\cap \GL(p_i+q_i,\BC)\\
   &=\SU(p_i,q_i)
\end{align*}
we have that
$$C_{\Sp(p,q)}(A)=\SU(p_1,q_1)\times\dots\times \SU(p_r,q_r)\times\Sp(p_s,q_s)\times\Sp(p_t,q_t)$$
meaning that
$$\dim S^A\le4p_sq_s+4p_tq_t+\sum_{i=1}^r 2p_iq_i$$
where $S=\Sp(p,q)/\Sp_p\times\Sp_q$. If $A$ is non-central we have that $p_s+q_s \neq p+q $ and $p_t+q_t \neq p+q$ and hence that
$$\dim S^A\le 4p(q-1).$$
Since we have by Proposition \ref{vcd lower bound} that $\vcd(\Gamma)\ge 4pq-p$ for any lattice $\Gamma\subset \Sp(p,q)$, we get the following from Corollary \ref{crit1}.

\begin{lem}\label{done-sppq}
Let $p\ge 2$. If  $\Gamma$ is a lattice in $\Sp(p,q)$ then $\vcd(\Gamma)=\cdm\Gamma$.\qed
\end{lem}

\subsection{Real symplectic group $\Sp(2n,\BR)$}\label{realsymplectic}
The maximal compact subgroup of $\Sp(2n,\BR)$ is $\UU_n$, which arises by first identifying $\BR^{2n}$ with $\BC^n$ and then setting
\begin{equation}\label{eq-copen4}
\UU_n=\Sp(2n,\BR)\cap\GL(n,\BC).
\end{equation}
This means that the eigenspaces of any $A\in\UU_n$ are complex and hence symplectic. Any such $A$ can be conjugated to have the form 
\begin{equation}\label{eq-copen3}
A=\left(
\begin{array}{ccc}
\lambda_1\Id_{d_1} & & \\
& \ddots & \\
& & \lambda_r\Id_{d_r}\end{array}
\right)
\end{equation}
where $\Id_{d_k}$ is a $d_k$-by-$d_k$ complex identity matrix and $\lambda_j\neq\lambda_k$ for all $j\neq k$. Note that $n=\sum_{i=1}^rd_i$. From here we get that
\begin{equation}\label{boundcentsymplectic}
C_{\Sp(2n,\BR)}(A)=C_{\Sp(2d_1,\BR)}(\lambda_1\Id_{d_1})\times\dots\times C_{\Sp(2d_r,\BR)}(\lambda_r\Id_{d_r}).
\end{equation}
Now, observing that $C_{\GL(2d_i,\BR)}(\lambda_i\Id_{d_i})=\GL(d_i,\BC)$ whenever $\lambda_i\notin\BR$ we get from \eqref{eq-copen4} that $C_{\Sp(2d_i,\BR)}(\lambda_i\Id_{d_i})\subset\UU_{d_i}$. This means that the only non-compact factors in $C_{\Sp(2n,\BR)}(A)$ correspond to the eigenvalues $\pm 1$. Altogether we have that 
\begin{equation}\label{eq-copen6}
\dim S^A=t^2+t+s^2+s
\end{equation}
where $t$ and $s$ are the complex dimensions of the $\pm 1$-eigenspaces and where $S=\Sp(2n,\BR)/\UU_n$. It follows that for any $A\in\Sp(2n,\BR)$ of finite order and non-central, i.e. if both $s,t<n$, this quantity is maximized if and only if either $s$ or $t$ is equal to $n-1$ and the other to $1$. In other words, we have that
$$\dim S^A\le n^2-n+2.$$
Since Proposition \ref{vcd lower bound} yields $\vcd(\Gamma)\ge n^2$, we obtain the following from Corollary \ref{crit1}.

\begin{lem}\label{done-sp2nr}
If $n\ge 3$ and $\Gamma\subset\Sp(2n,\BR)$ is a lattice, then $\vcd(\Gamma)=\cdm\Gamma$.\qed
\end{lem}

\subsection{The group $\SO^*(2n)$}
Recall that 
$$\SO^*(2n)=\{A\in\SU(n,n)\vert A^tQ_{n,n}J_nA=Q_{n,n}J_n\}$$
has as maximal compact subgroup
$$\left\{
\left(
\begin{array}{cc}
A  &  \\
  & \bar A
\end{array}
\right)\middle| A\in\UU_n\right\}\simeq\UU_n.$$
Similar considerations as before imply that if $A\in\UU_n$ has complex non-real eigenvalues $\lambda_1,\bar\lambda_1,\dots,\lambda_r,\bar\lambda_r$ then
$$C_{\SO^*(2n)}(A)\simeq\SO^*(2d_1)\times\dots\times\SO^*(2d_r)\times\SO^*(s)\times\SO^*(t)$$
where $d_i=\dim_\BC E_{\lambda_i}$ is the complex dimension of the eigenspace corresponding to the eigenvalue $\lambda_i$ and where $s$ (resp. $t$) is the dimension of the eigenspaces corresponding to the eigenvalue $1$ (resp. $-1$). This means that 
$$\dim S^A=-n+\frac{s^2}4+\frac{t^2}4+\sum_{i=1}^rd_i^2$$
where this time $S=\SO^*(2n)/\UU_n$. As usual this is maximized (for non-central $A$) by a sum with as few factors as possible, meaning that 
$$\dim S^A\le 1+(n-1)^2-n=n^2-n-2(n-1).$$
Since $n\ge 2$ by assumption, and since Proposition \ref{vcd lower bound} yields $\vcd(\Gamma)\ge n^2-n-\frac n2$, we get the following from Corollary \ref{crit1}.

\begin{lem}\label{done-noname}
For every lattice $\Gamma$ in $\SO^*(2n)$ we have $\vcd(\Gamma)=\cdm\Gamma$.
\end{lem}

%
%
%
%

\section{Real special linear group $\SL(n,\BR)$}\label{sec:slnr}
After the results in the previous section there are two families of classical groups to be dealt with: $\SL(n,\BR)$ and $\SO(p,q)$. Since the discussion is longer for these two classes of groups, we treat them in different sections. We prove now that $\cdm(\Gamma)=\vcd(\Gamma)$ for any lattice $\Gamma$ in $\SL(n,\BR)$ for $n\ge 3$. Note that all these lattices are arithmetic because $\rank_\BR\ge 2$.
\medskip

We start in the same way that we did in each of the particular cases in Section \ref{sec:easy}. Every $A\in\SL(n,\BR)$ that is non-central and has finite order can be conjugated into the maximal compact group $\SO_n$. Now, the same considerations as earlier yield that there are $d_t,d_s,d_1,\dots,d_r$ with $d_t+d_s+2d_1+\dots+2d_r=n$ such that 
$$C_{\SL(n,\BR)}(A)=S\left(\GL(d_t,\BR)\times\GL(d_s,\BR)\times\prod_{i=1}^r\GL(d_i,\BC)\right)$$
where $d_{\pm}$ is the dimension of the eigenspace of $A$ associated to the eigenvalue $\pm 1$; compare with Lemma \ref{centralizer-SO}. From here we get that
\begin{equation}\label{ivegotahangover}
\dim S^A=-1+\frac{d_t(d_t-1)}2+d_t+\frac{d_s(d_s-1)}2+d_s+\sum_{i=1}^rd_i^2
\end{equation}
where $S=\SL(n,\BR)/\SO_n$ is the relevant symmetric space. Moreover, since $A$ is non-central we have that $d_t,d_s\neq n$. This means that this formula takes its maximal values if $d_t=n-1$, $d_s=1$ and $r=0$. We have thus that
\begin{equation}\label{bad}
\dim S^A\le\frac{(n-1)(n-2)}2+n-1=\frac{n(n-1)}2
\end{equation}
with equality if and only if $A$ is conjugate to 
$$A\stackrel{\hbox{\tiny conj}}\simeq\left(
\begin{array}{cc}
-\Id_{n-1} &  \\
  & 1
\end{array}
\right)=Q_{n-1,1}$$
We deduce the following.

\begin{lem}\label{slkor1}
Let $\Gamma\subset\SL(n,\BR)$ be a lattice such that either $\Gamma$ does not contain any element conjugate to $Q_{n-1,1}$  or $\vcd(\Gamma)>\frac{n(n-1)}2$. Then $\cdm\Gamma=\vcd(\Gamma)$. 
\end{lem}
\begin{proof}
From \eqref{bad} we get that $\dim S^A\le\frac{n(n-1)}2$ for every $A\in\Gamma$ non-central. Moreover the inequality is strict if $A$ is not conjugate to $Q_{n-1,1}$. As for any lattice $\Gamma$ we have $\vcd(\Gamma)\ge\frac{n(n-1)}2$ by Proposition \ref{vcd lower bound}, we get under either condition in Lemma \ref{slkor1} that $\dim S^A<\vcd(\Gamma)$ for all $A\in\Gamma$ non-central and of finite order. The claim then follows from Corollary \ref{crit1}.
\end{proof}

We will apply Corollary \ref{cor: cohom flat} to deal with the case that $\vcd(\Gamma)=\frac{n(n-1)}2$, but first we need some preparatory work. As a first step, we bound the dimension of the intersection of distinct fixed point sets.
Our aim is to check that the hypothesis of Corollary \ref{cor: cohom flat} hold true so we only have to consider fixed point subsets in
$$\CS=\{S^H \ \vert\  H\in\CF\text{ non-central and } \nexists H'\in\CF\ \hbox{ non-central with}\ S^H\subsetneq S^{H'}\}.$$ 
Observe that for any $S^H\in\CS$ we have $S^H=S^A$ for some non-central finite order element $A$ and that if we have distinct $S^A,S^B\in\CS$ then $S^A\nsubseteq S^B$ and $S^B\nsubseteq S^A$.

\begin{lem}\label{intersection-strata-sl}
Suppose that $n\ge 4$. If $S^A,S^B$ are distinct elements in $\mathcal{S}$, then $$\dim S^{\langle A,B\rangle}=\dim(S^A\cap S^B)\le\frac{n(n-1)}2-2$$ 
where $\langle A,B\rangle$ is the subgroup of $\SO_n$ generated by $A$ and $B$.
\end{lem}
\begin{proof} Assume first that both $A$ and $B$ are conjugate to $Q_{n-1,1}$. Consider then the direct sum decompositions $\BR^n=V\oplus L$ and $\BR^n=V'\oplus L'$ whose factors are the eigenspaces of $A$ and $B$ respectively, so that $\dim L=\dim L'=1$. Note that, as long as $S^A\cap S^B\neq\emptyset$, we have $L\neq L'$ and $V\neq V'$ because $A\neq B$. Any matrix which commutes with $A$ and with $B$ has to preserve the direct sum decomposition
$$\BR^n=(V\cap V')\oplus L\oplus L'$$
which means that the common centralizer $C_{\SL(n,\BR)}(\langle A,B\rangle)$ of $A$ and $B$ is conjugate to $S(\GL(n-2,\BR)\times\BR^*\times\BR^*)$. This implies that 
$$\dim S^{\langle A,B\rangle}=\frac{n(n-1)}{2}+2-n\leq \frac{n(n-1)}2-2$$
as we claimed.

Assume now that $A$ is not conjugate to $Q_{n-1,1}$. Then we get from \eqref{ivegotahangover} that
$$\dim S^A\leq \frac{n(n-1)}2-1.$$
As $S^A\nsubseteq S^B$, $S^{\langle A,B\rangle}=S^A\cap S^B\subsetneq S^A$. Then the fact that 
$S^{\langle A,B\rangle}$ is a closed submanifold of  the connected manifold $S^A$ 
implies that 
$$\dim S^{\langle A,B\rangle}<\dim S^A\leq \frac{n(n-1)}2-1$$
so we are done.
\end{proof}
We remark that the lemma above also holds for matrices in $\mathrm{O}_n$. \\

Equation \eqref{bad} and Lemma \ref{intersection-strata-sl} imply that the first two conditions in Corollary \ref{cor: cohom flat} are satisfied as long as $n\ge 4$. To be able to check the third condition we need to construct appropriate flats in the symmetric space $S=\SL(n,\BR)/\SO_n$. Identify $S$ with the space of unimodular positive definite quadratic forms on $\BR^n$, suppose that $A\in\SL(n,\BR)$ is conjugate to $Q_{n-1,1}$, and let 
$$\BR^n=L\oplus V,\ \ \dim L=1\ \hbox{and}\ \dim V=n-1$$
be the decomposition of $\BR^n$ as direct sum of the eigenspaces of $A$. The fixed point set $S^A$ of $A$ consists of those unimodular scalar products on $\BR^n$ with respect to which the line $L$ is orthogonal to the subspace $V$. To describe a maximal flat start with a direct sum decomposition
$$\BR^n=P_1\oplus\dots\oplus P_n$$
where each $P_i$ has dimension $1$. The flat $F$ is the set of unimodular forms with respect to which the lines $P_1,\dots,P_n$ are mutually orthogonal. Note that general position of the direct sum decompositions $L\oplus V$ and $P_1\oplus\dots\oplus P_n$ implies that the corresponding fixed point set $S^A$ and flat $F$ are also in general position with respect to each other. 

\begin{prop}\label{prop:keysl}
Suppose that $A_1,\dots,A_r\in\SL(n,\BR)$ are such that $A_i$ is conjugate to $Q_{n-1,1}$ for all $i$ and such that $S^{A_i}\neq S^{A_1}$ for all $i\neq 1$. Then there is a maximal flat $F\subset S$ and an open neighborhood $U$ of $\Id\in\SL(n,\BR)$ such that for all $g\in U$ we have 
\begin{enumerate}
\item $gF$ intersects $S^{A_1}$ transversely in a point, and
\item $gF$ is disjoint from $S^{A_i}$ for $i\in\{2,\dots,r\}$.
\end{enumerate}
\end{prop}
\begin{proof}
To begin with note that $F$ and $S^{A_1}$ have complementary dimensions. In particular, to have a single transversal intersection is stable under small perturbations. In other words, (1) is automatically satisfied by every $gF$ with $g$ in some open neighborhood of $\Id$ if it is satisfied by $F$ itself. 

To find a maximal flat $F$ which, even after small perturbations, is disjoint from $S^{A_i}$ for $i=2,\dots,r$ we will use the following observation.

\begin{lem}\label{lem-la}
Suppose $V$ is a hyperplane in $\BR^n$ and let $u,v_1,\dots,v_n$ be vectors in the same component of $\BR^n\setminus V$. Suppose also that $v_1,\dots,v_n$ are linearly independent and that $u$ does not belong to the convex cone $\BR_+v_1+\dots +\BR_+v_n$. Then there is no scalar product with respect to which $u$ is orthogonal to $V$ and $v_1,\dots,v_n$ are orthogonal to each other.
\end{lem}
\begin{proof}
Suppose that there is such a scalar product and note that up replacing the vectors by positive multiples we may assume without loss of generality that all the vectors $u,v_1,\dots,v_n$ have unit length. Let $E$ be the hemisphere of $\BS^n\setminus V$ centred at $u$. Now, each $v_i\in E$ and this means that $u$ belongs to the ball in $\BS^n$ of radius $\frac\pi 2$ centred at $v_i$. In other words, $u$ belongs to the spherical triangle with vertices $v_i$ which means exactly that $u$ belongs to the cone $\BR_+v_1+\dots +\BR_+v_n$.
\end{proof}

Continuing with the proof of Proposition \ref{prop:keysl} and with the notation of Lemma \ref{lem-la}, suppose that $A\in\SL(n,\BR)$ is conjugate to $Q_{n-1,1}$ with associated direct sum decomposition $\BR^n=V\oplus\BR u$, and let $F$ be the flat in $S$ corresponding to the direct sum decomposition $\BR v_1\oplus\dots\oplus\BR v_n$. The non-existence of a scalar product with respect to which these two direct sum decompositions are orthogonal implies that $S^A\cap F=\emptyset$. In fact, noting that once $V$ and $u$ are fixed, the condition in Lemma \ref{lem-la} on the vectors $v_1,\dots,v_n$ is open, we also get that $S^A\cap gF=\emptyset$ for every $g\in\SL(n,\BR)$ near $\Id$.

Suppose now that $A_1,\dots,A_r$ are as in the statement of the proposition, let $\BR^n=V_i\oplus L_i$ be the associated direct sum decompositions with $\dim L_i=1$ and let $u_i\in L_i$ be a non-trivial vector. The upshot of the preceding discussion is that Proposition \ref{prop:keysl} follows once we can find a basis $v_1,\dots,v_n$ of $\BR^n$ in general position with respect to $V_i\oplus L_i$ for all $i=1,\dots,r$ such that
\begin{itemize}
\item[(i)] there is a scalar product with respect to which both direct sum decompositions $\BR v_1\oplus\dots\oplus\BR v_n$ and $V_1\oplus L_1$ are orthogonal, but that
\item[(ii)] for each $i=2,\dots,r$ we have that $u_i\notin\epsilon_1^i\BR_+v_1+\dots+\epsilon_n^i\BR_+v_n$ where $\epsilon_j^i=\pm 1$ is chosen so that $u_i$ and $\epsilon_j^iv_j$ are on the same side of $\BR^n\setminus V_i$.
\end{itemize}
To find the desired basis consider the projection $\pi:\BR^n\to V_1$ with kernel $L_1$. We choose a line $\ell\subset V_1$ with $\ell\neq\pi(L_i)$ for $i=2,\dots,r$. Moreover, if $V_i\neq V_1$ we assume that $\ell\not\subset V_i$.

 Up to a change of coordinates we can assume that $L_1$ and $V_1$ are orthogonal with respect to the standard scalar product $\langle\cdot,\cdot\rangle$, that $\ell$ is spanned by the first vector $e_1$ of the standard basis, and that $L_1$ is spanned by $e_2$. For $\lambda>0$ consider the linear map $\Phi_\lambda:\BR^n\to\BR^n$ whose matrix with respect to the standard basis $e_1,\dots,e_n$ is given by the diagonal matrix
 $$\Phi_\lambda=
\left(
\begin{array}{ccccc}
\lambda^2  &   & & &  \\
  & \lambda & &  &   \\
  &   &  1 & & \\
  & & & \ddots & \\
  & & & & 1 
\end{array}
\right)$$
and let $\langle\cdot,\cdot\rangle_\lambda$ be the unique scalar product with $\Phi_\lambda^*\langle\cdot,\cdot\rangle_\lambda=\langle\cdot,\cdot\rangle$. Choose now an $\langle\cdot,\cdot\rangle$-orthonormal basis $v_1^1,\dots,v_n^1$ in general position with respect to the direct sum decompositions $\BR^n=V_i\oplus L_i$ for $i=1,\dots,r$ and with 
$$\langle v_j^1,e_1\rangle>0\ \ \hbox{for all}\ j=1,\dots,n.$$
For $\lambda>1$ sufficiently large the $\langle\cdot,\cdot\rangle_\lambda$-orthonormal basis $v_1^\lambda,\dots,v_n^\lambda$ given by
$$v_j^\lambda=\Phi_\lambda(v_j^1)$$
has the desired properties. We leave the details to the reader.
\end{proof}

We are now ready to conclude the proof that Bredon cohomological dimension equals virtual cohomological dimension for lattices in $\SL(n,\BR)$ as long as $n\ge 4$.

\begin{lem}\label{lem:doneslnr}
If $n\ge 4$, then $\cdm(\Gamma)=\vcd(\Gamma)$ for every lattice $\Gamma\subset\SL(n,\BR)$.
\end{lem}
\begin{proof} Assume $n\geq 4$. From Lemma \ref{slkor1} we already know that the claim holds for all lattices $\Gamma\subset\SL(n,\BR)$ unless possibly if 
$$d=\vcd(\Gamma)=\frac{n(n-1)}2.$$
Suppose from now on that this is the case, and note that Proposition \ref{vcd lower bound} implies that $\rank_\BQ(\Gamma)=\rank_\BR(\SL(n,\BR))$, meaning that maximal rational flats are actually maximal flats.

From \eqref{bad} and from Lemma \ref{intersection-strata-sl} we get that the first two conditions in Corollary \ref{cor: cohom flat} are satisfied. We claim that also the third condition holds as well. To check that this is the case, suppose that we have $A_1,\dots,A_r\in\Gamma$ of finite order, non-central and with $\dim(S^{A_i})=\vcd(\Gamma)$ for all $i$ and with $S^{A_i}\neq S^{A_j}$ for all $i\neq j$. The discussion preceding Lemma \ref{slkor1} implies that each $A_i$ is conjugate in $\SL(n,\BR)$ to $Q_{n-1,1}$. Proposition \ref{prop:keysl} implies that there is a maximal flat $F$ and an open neighborhood $U$ of $\Id\in\SL_n\BR$ such that for all $g\in U$ we have 
\begin{enumerate}
\item $gF$ intersects $S^{A_1}$ transversely in a point, and
\item $gF$ is disjoint of $S^{A_i}$ for $i\in\{2,\dots,r\}$.
\end{enumerate}
Now, Lemma \ref{density flats} implies that we can find $g\in U$ such that $gF$ is rational. This proves that the third condition in Corollary \ref{cor: cohom flat} is also satisfied, and thus that $\vcd(\Gamma)=\cdm(\Gamma)$.
\end{proof}

The argument we just used does not apply to the case of lattices in $\SL(3,\BR)$ because there are lattices $\Gamma\subset \mathrm{SL}(3,\mathbb{R})$ which contain $A,B\in\SO_3\cap\Gamma$ with $S^A\neq S^B$ and $\dim(S^A\cap S^B)=2$ while $\vcd(\Gamma)=3$. In other words, the second condition in Corollary \ref{cor: cohom flat}  is not satisfied. We deal with the case of lattices in $\SL(3,\BR)$ now.

\begin{lem}\label{donesl3r}
If $\Gamma$ is a lattice in $\SL(3,\BR)$, then we have $\cdm(\Gamma)=\vcd(\Gamma)$.
\end{lem}
\begin{proof}
From Lemma \ref{slkor1} we get that the claim holds true except possibly if $\vcd(\Gamma)=3$, which means that the $\mathbb{Q}$-rank of $\Gamma$ is $2$. From Proposition \ref{prop-rank2} we get that $\Gamma$ contains a subgroup commensurable to $\SL(3,\BZ)$ or to $\SO(2,3)_\BZ$. In fact, we deduce from \cite[Prop. 6.44 and Prop. 6.48]{Witte} that $\Gamma\subset\SL(3,\BR)$ can be conjugated to a lattice commensurable to $\SL(3,\BZ)$. We may thus suppose that $\Gamma$ and $\SL(3,\BZ)$ were commensurable to begin with. This implies that the image of $\Gamma$ in the center free group $\PSL(3,\BC)$ is contained in $\PSL(3,\BQ)$, which amounts to saying that every element of $\Gamma$ is a real multiple of a matrix in $\GL(3,\BC)$ with rational entries. In particular, the direct sum decomposition $\BR^3=V\oplus L$ corresponding to any element $A\in\Gamma$ conjugate to $Q_{2,1}$ is rational. This implies that the 1-parameter group $\BT$ of $\SL(3,\BC)$ consisting of linear transformations which restrict to homotheties of $L\otimes_\BR\BC$ and $V\otimes_\BR\BC$ is a $\BQ$-split torus. Since on the other hand $\BT$ centralizes $A$, it follows that the (reductive) algebraic group $C_{\SL(3,\BC)}(A)$ has $\BQ$-rank at least 1, which yields that the set of integral points $C_{\SL(3,\BZ)}(A)$ does not act cocompactly on $S^A$, where $S$ is the symmetric space of $\SL(3,\BR)$. Since $\Gamma$ and $\SL(3,\BZ)$ are commensurable we get that  $C_{\Gamma}(A)$ does not act cocompactly on $S^A$ either. From the remark following Proposition \ref{thin-thick} we deduce that $\gdim C_\Gamma(A)<\dim S^A$ and hence $\cdm C_\Gamma(A)\le 2$ for every element $A\in\Gamma$ which is conjugate in $\SL(3,\BR)$ to $Q_{2,1}$. Since $\dim S^A\le 2$ for every other finite order element $1\neq A\in\Gamma$ we get that $\cdm C_\Gamma(A)\le 2<\vcd(\Gamma)$ for every non-trivial finite order element $A$. (Note that the center of $\Gamma$ is trivial.)

Let $\mathcal{F}$ be the collection of the finite subgroups of $\Gamma$.  It follows from the above that $\cdm C_\Gamma(H)\le 2<\vcd(\Gamma)$ for every non-trivial $H \in \mathcal{F}$. For $H \in \mathcal{F}$ non-trivial, consider $\underline{E}C_{\Gamma}(H)=X^H$ and the singular set for the $C_{\Gamma}(H)$-action on $X^H$
\[     Y=\bigcup_{\substack{K \in \mathcal{F} \\ H \subsetneq K \subseteq C_{\Gamma}(H)}} X^K.  \]
Since $\underline{\mathrm{cd}}(C_{\Gamma}(H))\leq 2$, it follows from \cite[Th. 1.1]{DMP} that $\mathrm{H}^n_c(X^H,Y)=0$ for all $n\geq 3$. Now consider the singular set for the $\Gamma$-action on $X$
\[   X^H_{\mathrm{sing}}= \bigcup_{\substack{K \in \CF \\ H \subsetneq K \subseteq  \Gamma}} X^K.    \]
Let $K$ be a finite subgroup of $\Gamma$ that strictly contains $H$ but not contained in $C_{\Gamma}(H)$. We claim that $\dim X^K \leq 1$. To prove this, first note that there exist two non-commuting finite order elements of $A$ and $B$ of $\Gamma$ such that $X^K \subseteq X^{\langle A, B \rangle}$. So it suffices to show that $\dim X^{\langle A, B \rangle}\leq 1$. The group $\langle A, B \rangle$ must contain a non-central finite order element that is not conjugate to $Q_{2,1}$. Indeed, if $A$ and $B$ are both conjugate to $Q_{2,1}$ then they both have order $2$. Since $A$ and $B$ do not commute, the product $AB$ is not central and cannot have order two. Hence $AB$ cannot be conjugate to $Q_{2,1}$. One checks using (\ref{ivegotahangover}) that $\dim S^A\leq 1$ for any finite order non-central element $A$ that is not conjugate to $Q_{2,1}$. Hence, we have $\dim X^{\langle A,B \rangle}\leq \dim X^{AB}\leq 1$ as desired.  We conclude that $X^H_{\mathrm{sing}}$ can be obtained from $Y$ be adding cells of dimension at most $1$. This implies that 

\[    \mathrm{H}^2_c(X^{H}_{\mathrm{sing}}) \rightarrow \mathrm{H}^2_c(Y)   \]
is injective and hence that $\mathrm{H}^n_c(X^H,X^H_{\mathrm{sing}})=0$ for all $n\geq 3$. Proposition \ref{gen} now implies that $\cdm(\Gamma)=\vcd(\Gamma)$, as we needed to show.
\end{proof}

%
%
%
%

\section{Indefinite orthogonal group $\SO(p,q)$}\label{sec:sopq}

Using a similar strategy as in the case of lattices in $\SL(n,\BR)$, we prove now that lattices $\Gamma$ in $\SO(p,q)$ also satisfy $\cdm(\Gamma)=\vcd(\Gamma)$. Since the $\BR$-rank 1 case is taken care of by Corollary \ref{crit0}, we will assume all the time that $p\le q$ and $p=\rank_\BR\SO(p,q)\ge 2$, meaning that all lattices in question are arithmetic. Let $S=\SO(p,q)/ S(\OO_p\times\OO_q)$ be the relevant symmetric space and $X$ its Borel-Serre compactification. 
\medskip

Suppose that $A\in\SO(p,q)$ is non-central and has finite order. As always, we may conjugate $A$ into the maximal compact group $S(\OO_p\times\OO_q)$. Arguing as in the cases of $\SU(p,q)$ and $\Sp(p,q)$, and taking into account the description of centralizers in orthogonal groups provided by Lemma \ref{centralizer-SO} we obtain the following lemma whose proof we leave to the reader.

\begin{lem}\label{centralizer-sopq}
Suppose that $A\in S(\OO_p\times\OO_q)$. Then there are non-negative integers $p_s,q_s,p_t,q_t,p_1,q_1,\dots,p_r,q_r$ with
$$p=p_s+p_t+2\sum_{i=1}^rp_i,\ \ q=q_s+q_t+2\sum_{i=1}^rq_i$$
such that the centralizer $C_{\SO(p,q)}(A)$ of $A$ in $\SO(p,q)$ is conjugate to 
$$C_{\SO(p,q)}(A)=S(\OO(p_s,q_s)\times\OO(p_t,q_t)\times\UU(p_1,q_1)\times\dots\times\UU(p_r,q_r)).$$
Moreover, if $r=0$ and  $(p_s,q_s)=(0,1)$ then $A$ is conjugate within $\SO(p,q)$ to $Q_{n-1,1}$, where $n=p+q$. Finally, $A$ is central if and only if $r=0$ and either $p_t=q_t=0$ or $p_s=q_s=0$.\qed
\end{lem}

From Lemma \ref{centralizer-sopq} we get the following bound for the dimension of the fixed point set $S^A$
\begin{equation}\label{eq-copen7}
\dim S^A=p_sq_s+p_tq_t+2\sum_{i=1}^rp_iq_i.
\end{equation}
This formula remains true for any $A\in \OO_p\times\OO_q$.  Moreover as long as $A$ is non-central this quantity is maximized if and only if $r=0$ and one of the following holds:
  \begin{itemize}
\item[i)]  $(p_s,q_s)=(0,1)$,

\item[ii)]  $(p_s,q_s)=(p,q-1)$,

\item[iii)] $p=q$ and  $(p_s,q_s)=(1,0)$,

\item[iv)] $p=q$ and  $(p_s,q_s)=(p-1,p)$.
  \end{itemize}
  In cases ii), iii) and iv) $A$ has determinant $-1$. This means that
  \begin{equation}\label{badpq}
\dim S^A\le p(q-1)
\end{equation}
and for $A\in S(\OO_p\times\OO_q)$
there is an equality if and only if $A$ is conjugate to $Q_{n-1,1}$, where $n=p+q$. This is enough to prove the following.

\begin{kor}\label{sokor1}
Suppose $\Gamma\subset\SO(p,q)$ is a lattice which either does not contain elements conjugate to $Q_{n-1,1}$ or with $\vcd(\Gamma)>p(q-1)$. Then $\cdm \Gamma=\vcd(\Gamma)$.
\end{kor}
\begin{proof} Let $A$ be a non-central finite order element of $\Gamma$.
From \eqref{badpq} we get that $\dim S^A \le p(q-1)$ where the inequality is strict if $\Gamma$ does not contain elements conjugate to $Q_{n-1,1}$. Since for any lattice $\Gamma$ we have $\vcd(\Gamma)\ge p(q-1)$ by Proposition \ref{vcd lower bound}, we get that under either condition in the statement of the Corollary we have that $\dim S^A<\vcd(\Gamma)$. The claim now follows from Corollary \ref{crit1}.
\end{proof}

Using Corollary \ref{sokor1}, we can also deal with lattices in $\SO(2,2)$.

\begin{lem}\label{doneso22}
If $\Gamma\subset\SO(2,2)$ is a lattice, then $\vcd(\Gamma)\ge 3$ and hence $\cdm\Gamma=\vcd\Gamma$.
\end{lem}
\begin{proof}
If $\vcd(\Gamma) <3$, Proposition \ref{prop-vcd} implies that $\rank_\BQ(\Gamma)=2$. In turn, this means that $\Gamma$ contains a subgroup $\Gamma'$ commensurable to either $\SL(3,\BZ)$ or to $\SO(2,3)_\BZ$ by Proposition \ref{prop-rank2}. In the former case this implies that $\vcd(\Gamma')=\vcd(\SL(3,\BZ))=3$ and in the latter $\vcd(\Gamma')=\vcd(\SO(2,3)_\BZ)=4$. Since on the other hand we have $\vcd(\Gamma)\ge\vcd(\Gamma')$, we arrive at a contraction.
\end{proof}

From now on we assume that not only $2\le p\le q$, but also that $n=p+q\ge 5$. Corollary \ref{sokor1} covers many cases of lattices $\Gamma\subset\SO(p,q)$ but not all. We start the discussion of the remaining cases as we did for lattices in $\SL(n,\BR)$, namely by bounding the dimension of intersections of fixed point sets in
$$\CS=\{S^H \ \vert\  H\in\CF\text{ non-central and } \nexists H'\in\CF\ \hbox{ non-central with}\ S^H\subsetneq S^{H'}\}.$$ 

\begin{lem}\label{intersection-strata-so}
Let $A,B\in S(\OO_p\times\OO_q)$ be non central finite order elements with $S^A,S^B\in\CS$ distinct. Then $\dim S^{\langle A,B\rangle}\leq p(q-1)-2$  where $\langle A,B\rangle$ is the group generated by $A$ and $B$.
\end{lem}
\begin{proof} Assume first that one of the matrices $A$ or $B$ is not conjugated to $Q_{n-1,1}$. Assume that matrix is for example $A$. Then  we get from \eqref{eq-copen7} that $\dim S^{A}\le p(q-1)-1$ and hence, as $S^{\langle A,B\rangle}\subsetneq S^A$ is a closed submanifold of the connected manifold $S^A$ we have
$$\dim S^{\langle A,B\rangle}<\dim S^A\le p(q-1)-1$$ 
so we get the result.

So from now on we suppose that both $A$ and $B$ are conjugated to $Q_{n-1,1}$. If the matrix $AB$ does not have order $2$, then it has a non-real eigenvalue.  With the same notation as above this means that for $AB$, $r\ge 1$ and hence we get from \eqref{eq-copen7} that $\dim S^{AB}\le p(q-2)$ and hence that $\dim S^{\langle A,B\rangle}\le p(q-2)$. 

If $AB$ is an involution, then $A$ and $B$ commute and can be simultaneously diagonalized.  It follows that any matrix $C$ commuting with $A$ and $B$ preserves the coarsest direct sum decomposition of $\BR^{p+q}$ with the property that each factor is contained in eigenspaces of both $A$ and $B$. Since both $A$ and $B$ are different and conjugate to $Q_{n-1,1}$, we get that any such $C$ preserves two lines and a space of dimension $n-2$. So, it follows that
$$C_{\SO(p,q)}(\langle A,B\rangle)=\SO(p',q')$$
where $p'+q'=p+q-2$ and $p'\leq p$, $q'\leq q$. This means that 
\begin{align*}
\dim S^{\langle A,B\rangle}
&\le\max\{(p-2)q,(p-1)(q-1),p(q-2)\}\\
&=p(q-1)+\max\{-q+1,-p\}\le p(q-1)-2
\end{align*}
as we claimed.
\end{proof}

Equation \eqref{badpq} and Lemma \ref{intersection-strata-so} imply that that the first two conditions in Corollary \ref{cor: cohom flat} are satisfied for lattices $\Gamma\subset\SO(p,q)$ with $p+q\ge 5$. As was the case for lattices in $\SL(n,\BR)$, to check the third condition we need to be able to construct flats with certain intersection pattern. 

Again, we start giving a concrete description of fixed points sets and flats in the symmetric space $S$ associated to $\SO(p,q)$. Denote by $\CQ$ both the quadratic form $\CQ(v)=v^tQ_{p,q}v$ and the corresponding bilinear form on $\BR^n$ where $n=p+q$. Let $\Gr(p,\BR^n)$ be the Grassmannian of $p$-dimensional subspaces of $\BR^n$. The symmetric space $S=\SO(p,q)/S(\OO_p\times\OO_q)$ can be identified with the set of those $p$-dimensional planes restricted to which the form $\CQ$ is negative definite
$$S=\{W\in\Gr(p,\BR^n): \CQ\vert_W<0\}$$
Suppose that $A\in\SO(p,q)$ is conjugate to $Q_{n-1,1}$ and let 
$$\BR^n=V_A\oplus L_A,\ \ \dim V_A=n-1\ \hbox{and}\ \dim L_A=1$$
be the decomposition of $\BR^n$ as direct sum of the eigenspaces of $A$ (note that both are orthogonal with respect to $\CQ$). The fixed point set $S^A$ of $A$ consists of those $p$-dimensional subspaces $W\in S$ which are contained in $V_A$
$$S^A=\{W\in\Gr(p,\BR^n): \CQ\vert_W<0,\ W\subset V_A\}.$$
To describe maximal flats in $S$ notice that $\SO(1,1)$ is a $1$-dimensional $\BR$-split torus. The direct product $\SO(1,1)^p$ of $p$ copies of $\SO(1,1)$ is then a $p$-dimensional torus. Since $\SO(1,1)^p$ is also a subgroup of $\SO(p,q)$ and  $\SO(p,q)$ has real rank $p$, it is a maximal torus. For the sake of concreteness we construct an embedding of $\SO(1,1)^p$ into $\SO(p,q)$. Consider a $\CQ$-orthogonal direct sum decomposition 
$$\BR^n=P_1\oplus\dots\oplus P_p\oplus P_{\hbox{\tiny rest}}$$
where for $i=1,\dots,p$ we have $\dim P_i=2$ and $\CQ\vert_{P_i}$ has signature $(1,1)$. Then we have that
\begin{equation}\label{noideawhatname}
\SO(1,1)^p\simeq\SO(\CQ\vert_{P_1})\times\dots\times \SO(\CQ\vert_{P_1})\times\Id.
\end{equation}
The maximal flat in $S$ corresponding to the maximal $\BR$-split torus \eqref{noideawhatname} is the set of those $p$-dimensional spaces $W\in S$ with $\dim(W\cap P_i)=1$ for $i=1,\dots,p$. Notice that since $W\in S$ we have that $\CQ\vert_W<0$ and that this implies that $\CQ\vert_{W\cap P_i}<0$ for $i=1,\dots,p$. We summarize all this discussion in the following lemma.

\begin{lem}\label{fix-flatpq}
Suppose that $A\in\SO(p,q)$ is conjugate to $Q_{n-1,1}$ and let 
\begin{equation}\label{directsum1pq}
\BR^n=V_A\oplus L_A,\ \ \dim V_A=n-1\ \hbox{and}\ \dim L_A=1
\end{equation}
be the decomposition of $\BR^n$ as direct sum of the eigenspaces of $A$. Consider also a $\CQ$-orthogonal direct sum decomposition 
\begin{equation}\label{directsum2pq}
\BR^n=P_1\oplus\dots\oplus P_p\oplus P_{\hbox{\em\tiny rest}}
\end{equation}
with $\dim P_i=2$ and such that $\CQ\vert_{P_i}$ has signature $(1,1)$ for $i=1,\dots,p$ and let 
$$F=\{W\in S\vert \dim(W\cap P_i)=1\ \hbox{for}\ i=1,\dots,p\}$$
be the corresponding maximal flat in $X$. Then $S^A$ and $F$ intersect if and only for $i=1,\dots,p$ there is $v_i\in P_i\cap V_A$ with $\CQ(v_i)<0$.\qed
\end{lem}
To clarify the meaning of Lemma \ref{fix-flatpq}, consider the following consequence.

\begin{lem}\label{lem-lapq}
With the same notation as in Lemma \ref{fix-flatpq} suppose that the restriction $\CQ\vert_{P_1\cap V_A}$ of $\CQ$ to $P_1\cap V_A$ is positive definite. Then $F\cap S^A=\emptyset$.\qed
\end{lem}

We are now ready to prove the analogue of Proposition \ref{prop:keysl}.

\begin{prop}\label{prop:keyso}
Suppose that $A_1,\dots,A_r\in\SO(p,q)$ are such that $A_i$ is conjugate to $Q_{n-1,1}$ for all $i$ and such that $S^{A_i}\neq S^{A_j}$ for all $i\neq j$. Then there is a maximal flat $F\subset S$ and an open neighborhood $U$ of $\Id\in\SO(p,q)$ such that for all $g\in U$ we have 
\begin{enumerate}
\item $gF$ intersects $S^{A_1}$ transversely in a point, and
\item $gF$ is disjoint of $S^{A_i}$ for $i\in\{2,\dots,r\}$.
\end{enumerate}
\end{prop}

\begin{proof}
For $i=1,\dots,r$, let 
$$\BR^n=L_i\oplus V_i,\ \ \dim L_i=1$$ 
be the direct sum decomposition whose factors are the eigenspaces of $A_i$ and observe  $V_i\neq V_1$ for all $i\neq 1$ because $L_i$ is the $\CQ$-orthogonal complement of $V_i$ and we are assuming that $A_i\neq A_1$. Noting that the light cone (i.e. the set of isotropic vectors) of $\CQ$ on $V_1$ is not contained in finitely many linear subspaces we can find a non-zero vector $v_0\in V_1\setminus\cup_{i\neq 1}V_i$ with $\CQ(v_0)=0$. Fix also a non-zero vector $w\in L_1$. Then for all $a,b\in\BR$ we have
$$\CQ(av_0+bw)=a^2\CQ(v_0)+b^2\CQ(w)=b^2\CQ(w)\ge 0$$
with equality if and only if $b=0$. This implies that for each $i\ge 2$ and each non-zero $v\in\Span(v_0,w)\cap V_i$ one has $\CQ(v)>0$. It follows that the same is true for any other $v_1\in V_1$ sufficiently close to $v_0$. Since $v_0$ is in the closure of the set $\{\CQ<0\}$ we deduce that there is $v_1\in V_1$ such that
\begin{itemize}
\item $\CQ(v_1)<0$, and
\item $\CQ(v)>0$ for each non-zero $v\in\Span(v_1,w)\cap V_i$ and each $i\ge 2$.
\end{itemize}
Consider the 2-dimensional plane 
$$P_1=\Span(v_1,w)$$ and
choose further 2-dimensional planes $P_2,\dots,P_p$ such that $\CQ\vert_{P_j}$ has signature $(1,1)$ for each $j=2,\dots,p$, such that $P_i$ is $\CQ$-orthogonal to $P_j$ for all $i\neq j$, $i,j\in\{1,\dots,p\}$. Now choose $P_{\hbox{\tiny rest}}$ such that 
\begin{equation}\label{right-flat}
\BR^n=P_1\oplus\dots\oplus P_p\oplus P_{\hbox{\tiny rest}}
\end{equation}
 is a $\CQ$-orthogonal direct sum decomposition of $\BR^n$.
By construction \[P_2,\dots,P_p\subset V_1\] and $\CQ\vert_{P_1\cap V_1}<0$. In particular, it follows from Lemma \ref{fix-flatpq} that the flat $F$ associated to \eqref{right-flat} meets $S^{A_1}$, i.e.
$$F\cap S^{A_1}\neq\emptyset.$$
On the other hand, again by construction we have that $\CQ\vert_{P_1\cap V_i}>0$ for each $i\neq 1$. It follows hence from Lemma \ref{lem-lapq} that 
$$F\cap S^{A_i}=\emptyset\ \ \hbox{for all}\ i\neq 1.$$
All this means that the flat $F$ satisfies the two conditions in the statement of Proposition \ref{prop:keyso}. Moreover, by construction any other flat $gF$ sufficiently close to $F$ does as well. This finishes the proof.
\end{proof}

\begin{rmk*}{\rm Lemma \ref{intersection-strata-so} and Proposition \ref{prop:keyso} can be extended to the case when the relevant matrices lie  in $\OO_p\times\OO_q$ only. The proof of Lemma \ref{intersection-strata-so} needs only  minor modifications to include the cases when the matrices $A$ and $B$ are of types ii), iii) and iv).  Consider now Proposition \ref{prop:keyso}. If $A_1$ is of type ii) we can use the same proof  just interchanging the role played by the 1 and -1 eigenspaces. By the same reason cases iii) and iv) are symmetric so we may assume that $A_1\in\OO_p\times\OO_q$ corresponds to case, say, iii) and therefore that $p=q$ and  $(p_s,q_s)=(1,0)$. The proof of Proposition  \ref{prop:keyso} can be modified as follows: one checks that 
$$S^A=\{W_A\oplus L_A:W_A\in\Gr(p-1,\BR^n): \CQ\vert_{W_A}<0,\ W_A\subset V_A\}.$$
where $\BR^n=V_A\oplus L_A$ is the decomposition in eigenspaces for $A$. Take the element that was denoted $w$ in the proof of  Proposition  \ref{prop:keyso} to be in $V_A$ and take $P_2=L_A\oplus\Span(v_2)$ for a suitable $v_2\in V_A$. The rest of the subspaces $P_i$ can be taken in the same way as in the current proof.
}
\end{rmk*}

Armed with Proposition \ref{prop:keyso}, we come to the main result of this section.

\begin{lem}\label{lem:donesopq}
If $p+q\ge 5$ and $\Gamma\subset\SO(p,q)$ is a lattice, then $\gdim(\Gamma)=\vcd(\Gamma)$.
\end{lem}
\begin{proof}
The argument is basically the same as that to prove Lemma \ref{lem:doneslnr}. From Corollary \ref{sokor1} we already know that the claim holds for all lattices $\Gamma\subset\SO(p,q)$, except those with
$$d=\vcd(\Gamma)=p(q-1)$$
Suppose from now on that this is the case, and note that Proposition \ref{vcd lower bound} implies that $\rank_\BQ(\Gamma)=\rank_\BR(\SO(p,q))$, meaning that maximal rational flats are actually maximal flats. From \eqref{badpq} and from Lemma \ref{intersection-strata-so} we get that the first two conditions in Corollary \ref{cor: cohom flat} are satisfied and we claim that also the third condition holds as well. To check that this is the case, suppose that we have $A_1,\dots,A_r\in\Gamma$ of finite order, non-central, with $\dim(S^{A_i})=\vcd(\Gamma)$ for all $i$ and with $S^{A_i}\neq S^{A_j}$ for all $i\neq j$. Then (\ref{badpq}) implies that each $A_i$ is conjugate in $\SO(p,q)$ to $Q_{n-1,1}$. Now Proposition \ref{prop:keyso} asserts that there is a maximal flat $F$ and an open neighborhood $U$ of $\Id\in\SO(p,q)$ such that for all $g\in U$ we have 
\begin{enumerate}
\item $gF$ intersects $S^{A_1}$ transversely in a point, and
\item $gF$ is disjoint of $S^{A_i}$ for $i\in\{2,\dots,r\}$.
\end{enumerate}
Now, Lemma \ref{density flats} implies that we can find $g\in U$ such that $gF$ is rational. This proves that the third condition in Corollary \ref{cor: cohom flat} is also satisfied, and thus that $\vcd(\Gamma)=\cdm(\Gamma)$.
\end{proof}

\section{Proof of the main theorem} \label{sec: proof main}
We are now ready to finish proof of the  Main Theorem.

\begin{mthm}
We have $\gdim(\Gamma)=\vcd(\Gamma)$ for every lattice $\Gamma$ in a classical simple Lie group.
\end{mthm}

\begin{proof}
First recall that the claim holds true for all lattices in groups of real rank 1 by Corollary \ref{crit0}. We might thus restrict ourselves to groups of higher rank, getting hence that the involved lattices are arithmetic. For the convenience of the reader we present the list of the classical simple Lie groups with $\rank_\BR\ge 2$

\begin{align*}
&\SL(n,\BC)=\{A\in\GL(n,\BC)\vert \det A=1\} & & n\ge 3 \\
& \SO(n,\BC)=\{A\in\SL(n,\BC) \vert A^tA=\Id\} & & n\ge 5 \\
&\Sp(2n,\BC)=\{A\in\SL(2n,\BC) \vert A^tJ_nA=J_n\} & & n\ge 2 \\
&\SL(n,\BR)=\{A\in\GL(n,\BR)\vert \det A=1\} & & n\ge 3 \\
&\SL(n,\BH)=\{A\in\GL(n,\BH)\vert \det A=1\} & & n\ge 3 \\
&\SO(p,q)=\{A\in\SL(p+q,\BR)\vert A^*Q_{p,q}A=Q_{p,q}\}& & 2\le p\le q \\
&\SU(p,q)=\{A\in\SL(p+q,\BC)\vert A^*Q_{p,q}A=Q_{p,q}\}& & 2\le p\le q \\
&\Sp(p,q)=\{A\in\GL(p+q,\BH)\vert A^*Q_{p,q}A=Q_{p,q}\}& & 2\le p\le q \\
&\Sp(2n,\BR)=\{A\in\SL(2n,\BR)\vert A^tJ_nA=J_n\} & & n\ge 2 \\
&\SO^*(2n)=\{A\in\SU(n,n)\vert A^tQ_{n,n}J_nA=Q_{n,n}J_n\} & & n\ge 4
\end{align*}
Note now that it follows from Proposition \ref{vcd lower bound} that if $\Gamma$ is a lattice in a higher rank classical Lie group then $\vcd(\Gamma)\ge 3$ with the only possible exceptions of lattices in $\SO(2,2)$ and $\Sp(4,\BR)$. In fact, every lattice in $\SO(2,2)$ has $\vcd\ge 3$ by Lemma \ref{doneso22} and the same is true for lattices in $\Sp(4,\BR)$ because $\SO(2,2)$ and $\Sp(4,\BR)$ are isogenous. Since all the lattices under consideration have $\vcd\ge 3$, they all have $\cdm\ge 3$ and hence we get from Theorem \ref{Luck-Meintrup} that 
$$\gdim(\Gamma)=\cdm(\Gamma).$$
In particular, it suffices to prove  $\cdm(\Gamma)=\vcd(\Gamma)$ for every lattice in any of the groups in the list we just gave. Basically, we have already proved that this is the case. In fact, it suffices to go through the list and collect what we already know for each one of them.

\begin{itemize}
\item $\SL(n,\BC)$: Lemma \ref{done-slnc} asserts that if $\Gamma$ is a lattice in $\SL(n,\BC)$ for $n\geq 2 $, then $\cdm\Gamma=\vcd(\Gamma)$.
\item $\SO(n,\BC)$: Lemma \ref{done-sonc} asserts that if $\Gamma\subset\SO(n,\BC)$ is a lattice where $n\ge 5$, then $\cdm\Gamma=\vcd(\Gamma)$.
\item $\Sp(2n,\BC)$: Lemma \ref{done-sp2nc} asserts that if $\Gamma$ is a lattice in $\Sp(2n,\BC)$ for $n\ge 2$, then we have $\cdm(\Gamma)=\vcd(\Gamma)$.
\item $\SL(n,\BR)$: Lemma \ref{lem:doneslnr} asserts that $\cdm(\Gamma)=\vcd(\Gamma)$ for every lattice $\Gamma\subset\SL(n,\BR)$, if $n\ge 4$. By Lemma \ref{donesl3r}, the same holds true for lattices in $\SL(3,\BR)$.
\item $\SL(n,\BH)$: Lemma \ref{done-slnh} asserts that for every lattice $\Gamma$ in $\SL(n,\BH)$ for $n\ge 3$, we have $\cdm(\Gamma)=\vcd(\Gamma)$.
\item $\SO(p,q)$: Lemma \ref{lem:donesopq} asserts that if $p+q\ge 5$ and if $\Gamma\subset\SO(p,q)$ is a lattice, then $\cdm(\Gamma)=\vcd(\Gamma)$. We also have $\cdm(\Gamma)=\vcd(\Gamma)$ for lattices in $\SO(2,2)$ by Lemma \ref{doneso22}.
\item $\SU(p,q)$: We have that $\cdm(\Gamma)=\vcd(\Gamma)$ for every lattice $\Gamma$ in $\SU(p,q)$ by Lemma \ref{done-supq}.
\item $\Sp(p,q)$: We also have that $\cdm(\Gamma)=\vcd(\Gamma)$ for every lattice $\Gamma$ in $\Sp(p,q)$ by Lemma \ref{done-sppq}.
\item $\Sp(2n,\BR)$: Lemma \ref{done-sp2nr} asserts that $\cdm(\Gamma)=\vcd(\Gamma)$ for every lattice $\Gamma\subset\Sp(2n,\BR)$, if $n\ge 3$. Since $\Sp(4,\BR)$ is isogenous to $\SO(2,3)$, it follows from Lemma \ref{doneso22} and Lemma \ref{lem:inv-finite extensions} that the same is true for every lattice in $\Sp(4,\BR)$.
\item $\SO^*(2n)$: We get that $\cdm(\Gamma)=\vcd(\Gamma)$ for every lattice $\Gamma\subset\SO^*(2n)$ from Lemma \ref{done-noname}.
\end{itemize}
Summing up, we have proved that if $\Gamma$ is a lattice in a classical simple Lie group then $\vcd(\Gamma)=\gdim(\Gamma)$, as claimed.
\end{proof}
We end this paper with the proof of Corollary \ref{cor}.
\begin{proof}
The real rank one case is dealt with by Proposition \ref{thin-thick} and Proposition \ref{borel-serre-vcd1}. Now consider the higher rank case, so $d=\mathrm{vcd}(\Gamma)\geq 3$. By the main theorem, the group $\Gamma$ admits a model for $\underline{E}\Gamma$ of dimension $d$. Moreover, the Borel-Serre bordification $X$ of $S=G/K$ is a cocompact model for $\underline{E}\Gamma$. The cellular chain complexes of all fixed points set $X^H$ for $H \in \mathcal{FIN}$ assemble to form a finite free resolution $C_{\ast}(X^{-}) \rightarrow \underline{\mathbb{Z}}$ in the category of $\mathcal{O}_{\mathcal{FIN}}\Gamma$-modules. Since $\underline{\mathrm{cd}}(\Gamma)=d$, there kernel of the map $C_{d-1}(X^{-}) \rightarrow C_{d-2}(X^{-})$ is a projective $\mathcal{O}_{\mathcal{FIN}}\Gamma$-module $P$. Since $C_{\ast}(X^{-})\rightarrow \underline{\mathbb{Z}}$ is a finite length resolution consisting of finitely generated free $\mathcal{O}_{\mathcal{FIN}}\Gamma$-modules, a standard trick in homological algebra involving Schanuel's lemma (see for example the discussion before \cite[Proposition 6.5]{brown}) implies that there exists a finitely generated free $\mathcal{O}_{\mathcal{FIN}}\Gamma$-module $F$ such that $F\oplus P$ is finitely generated and free, i.e.~$P$ is stably free. Hence, we have a resolution of length $d$
\[     0 \rightarrow    P\oplus F \rightarrow F \oplus C_{d-1}(X^{-}) \rightarrow C_{d-2}(X^{-})  \rightarrow \ldots \rightarrow C_{0}(X^{-})\rightarrow \underline{\mathbb{Z}} \rightarrow 0. \]
consisting of finitely generated free $\mathcal{O}_{\mathcal{FIN}}\Gamma$-modules. One can now apply the procedure of \cite[prop 2.5]{LuckMeintrup} and \cite[Th. 13.19]{Luck} to construct a cocompact model for $\underline{E}\Gamma$ of dimension $d$. Since all models for $\underline{E}\Gamma$ are $\Gamma$-equivariantly homotopy equivalent and the symmetric space $S=G/K$ is also a model for $\underline{E}\Gamma$, we conclude that $S$ is $\Gamma$-equivariantly homotopy equivalent to a cocompact $\Gamma$-CW-complex of dimension $\mathrm{vcd}(\Gamma)$.

\end{proof}


\begin{thebibliography}{99}

\bibitem{Armart} Aramayona, J. and Mart{\'\i}nez-P{\'e}rez, C.
 {\em The proper geometric dimension of the mapping class group}, Algebraic and Geometric Topology 14 (2014), 217--227
\bibitem{Ash}
Ash, A., {\em Small-dimensional classifying spaces for arithmetic subgroups of general linear 
groups}, Duke Math. J. 51 (1984), 459--468

\bibitem{BGS}
Ballmann, W., Gromov, M. and Schroeder, V., {\em Manifolds of non-positive curvature}, Progress in Math 61, Birkh\"auser, 1985

\bibitem{Borel-arith}
Borel, A., {\em Introduction aux groupes arithm\'etiques}, Hermann, (1969)

\bibitem{Borel20}
Borel, A., {\em Linear Algebraic Groups}, Proc. Symp. Pure Math. 9 (1969)

\bibitem{Borel-AG}
Borel, A., {\em Linear Algebraic Groups}, Springer (1991)



\bibitem{Borel-Serre}
Borel, A., and Serre, J.-P., {\em Corners and arithmetic groups}, Comment. Math. Helv. 48 (1973), 436--491
\bibitem{BradyLearyNucinkis} Brady, N., Leary I. , and Nucinkis, B..
    {\em On algebraic and geometric dimensions for groups with torsion}  J. Lond. Math. Soc.  Vol 64(2) (2001),  489--500
\bibitem{Bredon} Bredon, G.E.,
    {\em Equivariant cohomology theories} , Lecture Notes in Mathematics $34$, Springer ($1967$)
\bibitem{brown} Brown, K.S.,
{Cohomology of groups}, Graduate texts in Mathematics 77, Springer (1982)

\bibitem{DMP} Degrijse, D. and Mart{\'\i}nez-P{\'e}rez, C.,
    {\em Dimension invariants for groups admitting a cocompact model for proper actions}, Journal f\"{u}r Reine und Angewandte Mathematik: Crelle's Journal, doi: 10.1515/crelle-2014-0061
\bibitem{DMP2} Degrijse, D. and Mart{\'\i}nez-P{\'e}rez, C.,
    {\em Brown's question for finite extensions of right angled Coxeter groups}, preprint (2015)
\bibitem{Eberlein}
Eberlein, P., {\em Geometry of nonpositively curved manifolds}, Chicago Lectures in Math (1996)

\bibitem{Grayson}
Grayson, D.,{\em Reduction theory using semistability}, Comment. Math. Helv. 59 (1984),600--634

\bibitem{Helgason}
Helgason, S.,  {\em Differential geometry, Lie groups, and symmetric spaces}, Graduate Studies in Mathematics, 34. American Mathematical Society (1979) 

\bibitem{Ji-li}
Ji, L., {\em Lectures on locally symmetric spaces and arithmetic groups}, in {\em Lie groups and automorphic forms}, Studies in Advanced Mathematics 37, AMS (2006), 87--146

\bibitem{Ji-MacPherson}
Ji, L., and MacPherson, R., {\em Geometry of compactifications of locally symmetric spaces}, Annales de l'Institut Fourier 52, (2002), 457--559

\bibitem{Knapp}
Knapp, A., {\em Lie groups beyond an introduction.} Progress in Mathematics, 140. Birkh\"auser (2002)

\bibitem{KMPN} Kropholler, P., Mart\'inez-P\'erez, C. and Nucinkis, B., {\em Cohomological finiteness conditions for elementary amenable groups.} Journal f\"{u}r Reine und Angewandte Mathematik 637 (2009), 49--62

\bibitem{Leuzinger-BS}
Leuzinger, E., {\em An exhaustion of locally symmetric spaces by compact submanifolds with corners}, Invent. Math. 121 (1995), 389--410
\bibitem{LP}  Leary, I.J. and Petrosyan, N., {\em Groups with cocompact classifying spaces for proper actions and a question of K. S. Brown}, preprint


 \bibitem{LearyNucinkis} Leary, I.J.  and  Nucinkis, B.,
        {\em Some groups of type VF}, Invent. Math. 151 (1) (2003), 135--162
\bibitem{Luck} L\"{u}ck, W.,
    {\em Transformation groups and algebraic K-theory}, Lecture Notes in Mathematics, Vol. 1408, Springer-Berlin (1989)
    \bibitem{Luck1}   L\"{u}ck, W.,,
        {\em The type of the classifying space for a family of subgroups}, J. Pure Appl. Algebra 149 (2000), 177--203
\bibitem{Luck2}  L\"{u}ck, W.,
    {\em Survey on classifying spaces for families of subgroups}, Infinite Groups: Geometric, Combinatorial and Dynamical Aspects, Springer (2005), 269--322
\bibitem{LuckMeintrup} L\"{u}ck, W. and Meintrup, D.,
    {\em On the universal space for group actions with compact isotropy}, Proc. of the conference ``Geometry and Topology'' in Aarhus, (1998), 293--305
\bibitem{Margulis}
Margulis, G., {\em Discrete Subgroups of Semisimple Lie Groups}, Springer (1990)

\bibitem{Mostow}
Mostow, G., {\em Strong rigidity of locally symmetric spaces}, Ann. Math. Studies, Vol. 78, Princeton University Press (1973)

\bibitem{wr-minimal}
Pettet A. and Souto, J., {\em Minimality of the well-rounded retract}, Geom.\@Top. 12 (2008),1543--1556

\bibitem{Soule}
Soul\'e, C., {\em The cohomology of $\SL_3\BZ$}, Topology 17 (1978), 1--22
  \bibitem{Vogtmann} Vogtmann, K.
      {\em Automorphisms of free groups and outer space}, Geometriae Dedicata, Vol. 94 (1) (2002), 1--31


\bibitem{Witte}
Witte-Morris, D., {\em Introduction to Arithmetic Groups}, {arxiv math/0106063}


\end{thebibliography}
\end{document}